\documentclass[12pt]{amsart}
\usepackage{a4wide}
\usepackage[utf8]{inputenc}
\usepackage{amsmath}
\usepackage{amsfonts}
\usepackage{amssymb}
\usepackage{amsthm}
\usepackage{mathrsfs}
\usepackage{subfigure}
\usepackage[usenames,dvipsnames]{xcolor}
\usepackage{pstricks}
\usepackage{graphicx}
\usepackage[all]{xy}
\usepackage{varwidth}
\usepackage{enumerate}
\usepackage{bm}

\newtheorem{theo}{Theorem}[section]
\newtheorem{prop}[theo]{Proposition}
\newtheorem{coro}[theo]{Corollary}
\newtheorem{lemm}[theo]{Lemma}

\theoremstyle{definition}

\theoremstyle{remark}
\newtheorem{rema}[theo]{Remark}
\newcommand{\Op}{\operatorname{Op}}

\newcommand{\nwc}{\newcommand}
\nwc{\eps}{\epsilon}
\nwc{\ep}{\epsilon}
\nwc{\vareps}{\varepsilon}
\nwc{\Oph}{\operatorname{Op}_\hbar}
\nwc{\la}{\langle}
\nwc{\ra}{\rangle}

\nwc{\mf}{\mathbf} 
\nwc{\blds}{\boldsymbol} 
\nwc{\ml}{\mathcal} 

\nwc{\defeq}{\stackrel{\rm{def}}{=}}

\nwc{\cE}{\ml{E}}
\nwc{\cN}{\ml{N}}
\nwc{\cO}{\ml{O}}
\nwc{\cP}{\ml{P}}
\nwc{\cU}{\ml{U}}
\nwc{\cV}{\ml{V}}
\nwc{\cW}{\ml{W}}
\nwc{\tU}{\widetilde{U}}
\nwc{\IN}{\mathbb{N}}
\nwc{\IR}{\mathbb{R}}
\nwc{\IZ}{\mathbb{Z}}
\nwc{\IC}{\mathbb{C}}
\nwc{\IT}{\mathbb{T}}
\nwc{\IS}{\mathbb{S}}
\nwc{\tP}{\widetilde{P}}
\nwc{\tPi}{\widetilde{\Pi}}
\nwc{\tV}{\widetilde{V}}
\nwc{\supp}{\operatorname{supp}}
\nwc{\rest}{\restriction}
\let \d \relax
\nwc{\d}{\partial}
\nwc{\Cor}{\mathscr{C}}

\nwc{\todo}[1]{$\clubsuit$ {\tt #1}}

\begin{document}

\title[Quantum limits of perturbed sub-Riemannian Laplacians in dimension $3$]{Quantum limits of perturbed sub-Riemannian contact Laplacians in dimension $3$}

\author{V\'ictor Arnaiz}
\address{Laboratoire de Math\'ematiques Jean Leray, Nantes Universit\'e, UMR CNRS 6629, 2 rue de la Houssini\`ere, 44322 Nantes Cedex 03, France}

\email{victor.arnaiz@univ-nantes.fr}

\author{Gabriel Rivi\`ere}

\address{Laboratoire de Math\'ematiques Jean Leray, Nantes Universit\'e, UMR CNRS 6629, 2 rue de la Houssini\`ere, 44322 Nantes Cedex 03, France}

\address{Institut Universitaire de France, Paris, France}

\email{gabriel.riviere@univ-nantes.fr}

\begin{abstract} On the unit tangent bundle of a compact Riemannian surface, we consider a natural sub-Riemannian Laplacian associated with the canonical contact structure. In the large eigenvalue limit, we study the escape of mass at infinity in the cotangent space of eigenfunctions for hypoelliptic selfadjoint perturbations of this operator. Using semiclassical methods, we show that, in this subelliptic regime, eigenfunctions concentrate on certain quantized level sets along the geodesic flow direction and that they verify invariance properties involving both the geodesic vector field and the perturbation term.
\end{abstract}

\maketitle

\section{Introduction}

Let $(M,g)$ be a smooth, compact, oriented, and boundaryless Riemannian \emph{surface} and denote by $K(m)$ its sectional curvature at a given point $m\in M$. The unit tangent bundle of $M$ is defined by 
$$\mathcal{M}:=SM=\left\{q=(m,v)\in TM:\ \|v\|_{g(m)}=1\right\}.$$
There are two natural vector fields on $SM$: the geodesic vector field $X$ and the vertical vector field $V$, \emph{i.e.} the vector field corresponding to the action by rotation in the fibers of $SM$. One can then define $X_\perp:=[X,V]$ and these vector fields verify the following commutation relations~\cite[\S3.5.1]{PaternainSaloUhlmann22}:
\begin{equation}\label{e:commutators}
 [X,X_\perp]= -KV,\quad [X,V]=X_\perp,\quad\text{and}\quad [X_\perp,V]=-X,
\end{equation}
where $K$ is understood as a function on $SM$ (via pullback). The manifold $\mathcal{M}$ is naturally endowed with a Riemannian metric $g_S$ (the Sasaki metric) which makes $(X,X_\perp,V)$ into an orthonormal basis. The corresponding volume form that we will denote by $d\mu_L$ makes these three vector fields divergence free and we can define the sub-Riemannian Laplacian associated with this geodesic frame:
$$-\Delta_{\text{sR}}:=X_\perp^*X_\perp+V^*V=-X_\perp^2-V^2.$$
More precisely, we consider the Friedrichs extension of this formally selfadjoint operator (see Appendix~\ref{a:spectral} for a brief reminder) which is hypoelliptic by H\"ormander's Theorem~\cite[Th.~1.1]{Hormander67}. In the context of contact geometry, $-\Delta_{\text{sR}}$ is referred as the Rumin Laplacian for the Sasaki metric~\cite{Rumin1994}. See also~\S\ref{ss:general-contact} for a discussion on the case of general H\"ormander (contact) operators in dimension $3$.

We now let $Q, W\in\ml{C}^{\infty}(\ml{M},\IR)$. Our goal is to study, in the semiclassical limit $h\rightarrow 0^+$, the eigenfunctions of the following (formally selfadjoint) operator:
\begin{equation}
\label{e:main_operator}
\widehat{P}_h:=-h^2\Delta_{\text{sR}}-ih^2QX -\frac{ih^2X(Q)}{2}+W, \quad h \in (0,1].
\end{equation}
Under the assumption $\|Q\|_{\ml{C}^0}<1$, one can again consider the Friedrichs extension of this operator and, from the Rothschild-Stein Theorem~\cite[Th.~16]{RothschildStein76}, this still defines an \emph{hypoelliptic} operator. Combining this last Theorem with classical tools from spectral theory~\cite{ReedSimon80, ReedSimon75}, one can find $h_0 > 0$ such that, for all $0<h\leq h_0$, there exists a nondecreasing sequence 
$$\text{min}\ W+\ml{O}_Q(h)\leq \lambda_h(0)\leq \lambda_h(1)\leq\ldots \leq \lambda_h(j)\ldots \rightarrow +\infty,\quad\text{as}\quad j\rightarrow +\infty,$$ 
and an orthonormal basis $(\psi_h^j)_{j\geq 0}$ of $L^2(\mathcal{M})$ verifying, for all $j\geq 0$,
\begin{equation}\label{e:semiclassical-basis}
 \widehat{P}_h \, \psi_h^j =\lambda_h(j) \,\psi_h^j. 
\end{equation}
We refer to Lemma~\ref{l:spectrum} in Appendix~\ref{a:spectral} for details. Moreover, any solution $\psi_h^j$ to this eigenvalue problem belongs to the space $\ml{C}^{\infty}(\ml{M})$ and, thanks to Lemma~\ref{l:apriori-estimate}, it satisfies the a priori estimate for $h>0$ small enough:
$$\|hX_\perp\psi_h^j\|_{L^2}^2+\|hV\psi_h^j\|_{L^2}^2+\|h^2X\psi_h^j\|_{L^2}^2\leq C_{Q,W}(1+|\lambda_h(j)|)^2,$$
where $C_{Q,W}>0$ is a constant depending only on $(Q,W)$. Here, the fact that there is a factor $h^2$ in front of the derivatives of $X$ is a manifestation of the lack of ellipticity of the operator along the $X$ direction. In the following, we aim precisely at analyzing the structure of the eigenfunctions in the subelliptic regime where formally speaking one has $h^{-1}\ll |X|\lesssim h^{-2}$.

\subsection{Quantum limits and semiclassical measures}
We are interested in describing the asymptotic properties of the semiclassical eigenmodes satisfying\footnote{All along the article, we use the standard conventions from semiclassical analysis to write $h\rightarrow 0^+$ instead of writing a sequence $h_n\rightarrow 0$ as $n\rightarrow\infty.$}:
\begin{equation}\label{e:semiclassical-eigenvalue}
 \widehat{P}_h \, \psi_h =\lambda_h \, \psi_h,\quad\|\psi_h\|_{L^2}=1,\quad\lambda_h\rightarrow\lambda_0\in\IR,\ \text{as}\ h\rightarrow 0^+. 
\end{equation}
When $W\equiv 0$, a natural choice is to pick $\lambda_h=1$ that would correspond to studying the large eigenvalue limit for the hypoelliptic operator 
$\mathcal{L}=\Delta_{\text{sR}}+iQX +\frac{iX(Q)}{2}$. Yet, as we want to emphasize the semiclassical nature of this spectral problem, we keep a general $W$ and thus some general value $\lambda_0\geq\min W$. Still from Lemma~\ref{l:apriori-estimate}, one finds that, for any sequence $\lambda_h\rightarrow \lambda_0$, there exists some $h_0>0$ such that, for all $0<h\leq h_0$ and for any solution to~\eqref{e:semiclassical-eigenvalue},
\begin{equation}\label{e:apriori-semiclassical}
 \|hX_\perp\psi_h\|_{L^2}^2+\|hV\psi_h\|_{L^2}^2+\|h^2X\psi_h\|_{L^2}^2\leq C_{Q,W}(1+2|\lambda_0|)^2.
\end{equation}

One says that a probability measure $\nu$ is a \emph{quantum limit} for this spectral problem if, for every $a\in\ml{C}^0(\ml{M})$, 
$$\lim_{h\rightarrow 0^+}\int_{\ml{M}}a|\psi_h|^2d\mu_L=\int_{\ml{M}}a \, d\nu,$$
where $(\psi_h)_{h\rightarrow 0^+}$ is a sequence verifying~\eqref{e:semiclassical-eigenvalue}. Up to extraction of a subsequence, one can always find such an accumulation point. Given $\lambda_0\geq\min W$, we denote by $\ml{N}_{\lambda_0}$ the set of quantum limits associated with the spectral parameter $\lambda_0$. In view of describing the regularity properties of $\nu$, one lifts the problem to the cotangent bundle $T^*\ml{M}$ by introducing
$$w_h:a\in\ml{C}^{\infty}_c(T^*\ml{M})\mapsto\langle \Op_h(a)\psi_h,\psi_h\rangle_{L^2},$$
where $\text{Op}_h(a)$ is a $h$-pseudodifferential operator with principal symbol $a$~\cite[Th.~14.1]{Zworski12} and $(\psi_h)_{h\rightarrow 0^+}$ is the sequence used to generate $\nu$. Thanks to the Calder\'on-Vaillancourt Theorem~\cite[Th.~5.1]{Zworski12}, $(w_h)_{h\rightarrow 0^+}$ is a bounded sequence in $\ml{D}^\prime(T^*\ml{M})$. Hence, up to extraction, it converges to some limit $w$ which is referred as a semiclassical measure for the sequence $(\psi_h)_{h\rightarrow 0^+}$. The theory of semiclassical pseudodifferential operators allows to prove that any such $w$ is a \emph{finite nonnegative measure} on $T^*\ml{M}$ that is supported on
$$\ml{E}^{-1}(\lambda_0):=\{(q,p)\in T^*\ml{M}: \mathcal{E}(q,p) := H_2(q,p)^2+H_3(q,p)^2+W(q)=\lambda_0\},$$
and that satisfies the following invariance property
$$\left\{H_2^2+H_3^2+W,w\right\}=0,$$
where $$H_2(q,p):=p(X_\perp),\quad\text{and}\quad H_3(q,p):=p(V).$$
See for instance~\cite[\S5.2]{Zworski12} for proofs of these facts in the case of $\IR^{2d}$. We emphasize that, contrary to the case of eigenvalue problems of elliptic nature, the energy layer $\ml{E}^{-1}(\lambda_0)$ is not compact and there may be some escape of mass at infinity. In particular, $w$ could be equal to $0$. See for instance~\S\ref{s:torus} for concrete examples in the case of the flat torus. Due to this escape of mass at infinity, it is natural to study the measure
$$\nu_\infty:=\nu-\pi_*w,$$
where $\pi:(q,p)\in T^*\ml{M}\rightarrow q\in \ml{M}$, and this is the main purpose of the present work.

\subsection{Decomposition of the measure $\nu_\infty$ and invariance properties}

In~\cite[Thm. B]{ColindeVerdiereHillairetTrelat15}, Colin de Verdi\`ere, Hillairet and Tr\'elat proved that $X(\nu_\infty)\equiv0$ when $Q\equiv0$ and $W\equiv 0$. 
Our results generalize this Theorem in two directions. First, we will provide a refined description of $\nu_\infty$, showing that the measure $\nu_\infty$ decomposes into a discrete sum of non-negative Radon measures covering different asymptotic regimes $h^{-1}\ll |X|\lesssim h^{-2}$ across the non-compact part of $\mathcal{E}^{-1}(\lambda_0)$. Second, we will prove that each of these measures satisfies a new invariance property, different from each-other, as soon as $\nabla(Q)$ does not vanish. In view of formulating our results, we associate to each smooth function $f$ on $\ml{M}$ a natural vector field lying in the contact plane $D:=\text{Span}(X_\perp,V)$ given by
$$\Omega_f:=V(f)X_\perp-X_\perp(f)V.$$
Our main Theorem then reads:
\begin{theo}
\label{t:simplified_theorem}
Let $Q,W\in\ml{C}^\infty(\ml{M},\IR)$ such that $\|Q\|_{\ml{C}^0}<1$, let $\lambda_0 > \max_{q \in \mathcal{M}} W(q)$ and set
\begin{equation}\label{e:good-vectorfield}
 Y_W:=X+\Omega_{\ln(\lambda_0-W)}.
\end{equation}
Then, for every $\nu\in\ml{N}_{\lambda_0}$, the measure $\nu_\infty$ decomposes as 
\begin{equation}
\label{e:decomposition_nu}
\nu_\infty = \overline{\nu}_\infty + \sum_{k=0}^\infty \big( \nu^+_{k,\infty} + \nu_{k,\infty}^- \big),
\end{equation}
where $\overline{\nu}_\infty$ and $\nu_{k,\infty}^\pm$ are  non-negative Radon measures on $\mathcal{M}$ verifying, for all $a\in\ml{C}^1(\ml{M})$ and for all $k\in\mathbb{Z}_+$,
\begin{equation}\label{e:invariance-base}
\int_{\ml{M}}Y_W(a) \, d\overline{\nu}_\infty=0,\quad\text{and}\quad\int_{\ml{M}}Y_{W,Q,k}^\pm(a) \, d\nu_{k,\infty}^\pm=0,
\end{equation}
with 
$$Y^\pm_{W,Q,k} := \big( \pm (2k + 1) + Q  \big) Y_W-\Omega_Q.$$
\end{theo}

\begin{rema}
Condition $\lambda_0 > \max_{q \in \mathcal{M}} W(q)$ ensures that the classical forbidden region is empty. In the case $\min W \leq \lambda_0 \leq \max W$, the support of $\overline{\nu}_\infty$ becomes confined inside the compact set $\mathcal{M}_{\lambda_0,W} := \{ q \in \mathcal{M} \, : \, \lambda_0 - W \geq 0 \}$, while the support of $\nu_{k,\infty}^\pm$ is contained in the open subset $\mathcal{U}_{\lambda_0,W}:= \{ q \in \mathcal{M} \, : \, \lambda_0 - W > 0 \}$. This more general situation is covered by the more precise description of semiclassical measures stated in Theorem \ref{t:maintheo}.
\end{rema}

\begin{rema}
In Section \ref{s:torus}, by working on the flat torus $M = \mathbb{T}^2$, we show examples of sequences $(\psi_h,\lambda_h)$ satisfying \eqref{e:semiclassical-eigenvalue} for which the measures $\overline{\nu}_\infty$ or $\nu_{k,\infty}^\pm$ we construct carry the total mass of $\nu$. 
\end{rema}

Decomposition \eqref{e:decomposition_nu} reflects the stratification of the asymptotic phase-space distribution of a given sequence $(\psi_h,\lambda_h)$ satisfying \eqref{e:semiclassical-eigenvalue}. Indeed, in Section \ref{s:main_results_and_proofs} below, we will provide a more general description of $\nu_\infty$ by lifting the analysis to the phase-space via introducing an adapted semiclassical measure $\mu_\infty$ on  $\mathcal{M} \times \IR$ such that, by projection\footnote{Letting $\mu$ be a finite Radon measure on $\mathcal{M} \times \IR$, the measure $\nu(q) = \int_\IR \mu(q,dE)$ is defined by 
$$
\langle \nu , a \rangle := \int_{\mathcal{M} \times \IR} a(q) d\mu(q,E),
$$
for all $a \in \mathcal{C}^0(\mathcal{M})$.}:
$$
\nu_\infty(q) = \int_\IR \mu_\infty(q,dE).
$$ 
The extra variable $E \in \IR$ parameterizes the phase-space escape of mass along the degenerate direction of $X$ as $h \to 0^+$. We refer to~\eqref{e:full-wigner-H1} below for the explicit construction of the measure $\mu_\infty$ using semiclassical tools and we just give here some informal explanation. Letting $H_1(q,p) = p(X)$, we will study precisely two different asymptotic regimes  generating a splitting of the measure $\mu_\infty $ into two parts 
$$
\mu_\infty = \mathbf{1}_{E \neq 0} \mu_\infty + \mathbf{1}_{E = 0} \mu_\infty
$$ 
of qualitatively different nature:
\begin{itemize}
\item The \textbf{critical subelliptic regime} $h\vert H_1 \vert \asymp 1$, captured by $\mathbf{1}_{E \neq 0} \mu_\infty$, displays a quantum behavior which manifests as a decomposition of this measure into a discrete sum of Radon measures $(\mu^\pm_{k,\infty})_{k\in \mathbb{N}}$ supported on quantized level sets $\mathcal{H}_\pm^{-1}(2k+1) \subset \mathcal{M} \times \IR^*_\pm$ for the energy functions 
\begin{equation}\label{e:preserved-quantity}
\ml{H}_\pm(q,E) := \pm \left( \frac{\lambda_0-W(q)}{E} - Q(q) \right), \quad  (q,E) \in \mathcal{M} \times \IR^*_\pm.
\end{equation}
These measures project on the manifold $\mathcal{M}$ and give the measures $\nu^\pm_{k,\infty}$: 
$$
\nu^\pm_{k,\infty}(q) = \int_\IR \mu_{k,\infty}^\pm (q,dE).
$$
\medskip

\item The \textbf{subcritical subelliptic regime} $1 \ll \vert H_1 \vert \ll h^{-1}$, captured by the measure $\overline{\mu}_\infty = \mathbf{1}_{E = 0} \mu_\infty$, which is supported on $\mathcal{M} \times \{0 \}$. This measure projects on $\mathcal{M}$ so that:
$$
\overline{\nu}_\infty(q) = \int_\IR \overline{\mu}_\infty(q,dE).
$$
\end{itemize} 
Besides this distinction between the different oscillation regimes, our analysis will show the influence of the hypoelliptic perturbations of $-\Delta_{\text{sR}}$ given by \eqref{e:main_operator} in the previous description by obtaning new invariance properties of $\mu_\infty$ in terms of $Q$ and $W$. Among other things, it illustrates that the introduction of the new variable $E=hH_1$ becomes essential for this description even in the non-semiclassical set-up where $W\equiv 0$.

\subsection{More related results and questions}

 The fine analysis of these regimes $h\ll |E|=h|H_1|\lesssim 1$ in the subelliptic region of $T^*\ml{M}$ is reminiscent of the analysis made by Burq and Sun for the semiclassical measures of Baouendi-Grushin operators in~\cite{BurqSun22} (see also~\cite{LetrouitSun20, ArnaizSun22} for related works). More precisely, Theorem \ref{t:maintheo} below can be compared with the results obtained by the first author and Sun in~\cite{ArnaizSun22} where a detailed study of semiclassical measures in the subelliptic regime for quasimodes of the Baouendi-Grushin operator was performed. In these references, the operator is $\partial_x^2+a(x)\partial_y^2$, where $(x,y)\in\IT^2$ and $a(x)\geq 0$ is a smooth function that may vanish at finitely many isolated points (with nondegenerate tangencies). The role of $H_1$ is then played by the cotangent variable $\eta$ that is dual to $y$ and~\cite[\S3]{ArnaizSun22} gives a full description of the eigenmodes in the regime $1\ll |\eta|\lesssim h^{-1}$ through their semiclassical measures. In particular, the invariance of these measures through the vector field $\partial_y$ is shown and it replaces the geodesic vector field $X$ in that context. The result is actually stronger as they introduce operator-valued measures lifting the analogue of the measure $\nu_\infty$ and describe completely these lifted objects involving the full cotangent variables. Here, we only focus on the behaviour along the variables $(q,hH_1(q,p))$. Introducing the full cotangent variables would require some extra and delicate work (especially in the critical regime $h|H_1|\asymp 1$) that is not necessary to prove the results we are aiming at.

In~\cite{ColindeVerdiereHillairetTrelat15}, the hypoelliptic model is closer to ours but this extra variable did not appear in the description of $\nu_\infty$ because of the use of microlocal methods. The reason for introducing this new variable $E=hH_1$ is that, in the regime $E\neq 0$, the term $h^2QX$ is not negligible anymore compared with $-h^2\Delta_{\text{sR}}$ and it has to be taken into account in the description of the quantum limit. It results in new invariance properties as in Theorem~\ref{t:simplified_theorem}. The fact that the eigenfunctions are localized on specific levels is a manifestation of the fact that our hypoelliptic operators are modeled locally on the $3$-dimensional Heisenberg group (and thus related to the $1$-dimensional harmonic oscillator). More specifically, our proof of this support property will only rely on the fact that the sub-Riemannian Laplacian can be written as
\begin{equation}\label{e:sRLaplacian-harmonic}
\Delta_{\text{sR}}=Z^*Z-iX=ZZ^*+iX,\quad [Z,Z^*]=2iX,
\end{equation}
where $Z=X_\perp+i V$.

This quantization of the level sets can be thought as an analogue in our (non-algebraic) set-up of the decomposition appearing in the results of Fermanian-Kammerer and Fischer~\cite[Th.~1.1, Th.~2.10]{FermanianFischer21}. See also~\cite{FermanianLetrouit20} for related results in the compact setting. In these references, the decomposition of the semiclassical measures along these quantized levels shows up because there is a natural way to diagonalize the sub-Riemannian Laplacian along the elliptic variables. This is exactly where the harmonic oscillator appears in these references and the subelliptic variable $H_1$ corresponds exactly to the center direction of the Lie algebra setting from~\cite{FermanianFischer21, FermanianLetrouit20}. In these works, the proof of this decomposition required the introduction of operator-valued semiclassical measures. In the case of more general contact manifolds, we can also mention the works of Taylor~\cite{Taylor20} regarding the question of microlocal Weyl laws for operator-valued symbols. Here, our proof of these support properties will not rely on the introduction of such analytical objects. It will simply follow from a careful use of the relation~\eqref{e:sRLaplacian-harmonic} where $Z$ and $Z^*$ will play the role of ladder operators, in a similar way to the proof that the eigenvalues of the harmonic oscillator are given by $\{2k+1, k\geq 0\}$. In the general $3$-dimensional contact case, we note that the quantum normal form from~\cite{ColindeVerdiereHillairetTrelat15} as formulated in~\cite[\S6.2]{ColindeVerdiere2023} should in principle allow to get as in~\cite{FermanianFischer21} a natural decomposition of the measure $\nu_\infty$ using the spectral decomposition of the harmonic oscillator. Yet, due to its microlocal nature, it would not distinguish the various subelliptic regimes $1\ll |H_1|\lesssim h^{-1}$ involved in our problem as we are doing here.

As our semiclassical analysis of eigenfunctions for hypoelliptic operators is inspired by the fine analysis performed for the Baouendi-Grushin operator in~\cite{BurqSun22, ArnaizSun22}, it is natural to expect that such results would remain true for similar hypoelliptic perturbations of the Baouendi-Grushin operator. Similarly, the analysis presented here should in principle allow to deal with the controllability of the Schr\"odinger equation as in~\cite{BurqSun22} and with the stabilization of the wave equation as in~\cite{ArnaizSun22}. Yet, this would require more work that is beyond the scope of the present article. Another natural question would be to study more precisely the regularity of quantum limits when the geodesic vector field $X$ enjoys specific dynamical structure, e.g. on Zoll surfaces, on flat tori or on negatively curved surfaces. Among the natural questions to explore is whether one can always find sequences of eigenfunctions concentrating on a given levet set $\mathcal{H}_\pm^{-1}(2k+1)$. Related to this, it would be interesting to describe semiclassical Weyl laws for symbols involving the extra variable $E=hH_1$. In that direction, we refer one more time to~\cite{Taylor20} for microlocal Weyl laws with operator-valued symbols (including the case where $Q$ is not identically equal to $0$) on contact manifolds of dimension $\geq 3$. Finally, these hypoelliptic models are naturally related to semiclassical magnetic Schr\"odinger operators. For instance, in view of the works~\cite{RaymondVuNgoc15, HelfferKordyukovRaymondVuNgoc16, Morin19, Morin20}, it would be natural to compare how the fine structure of eigenfunctions of these models could be understood following the lines of the present work. Recall that rather precise descriptions of the low-energy eigenfunctions were already given via WKB and normal form methods in~\cite{BonthonneauRaymond20, GuedesBonthonneauRaymondVuNgoc21,  GuedesBonthonneauNguyenRaymondVuNgoc21}.

\subsection{A few words on more general sub-Riemannian contact Laplacians in dimension $3$}\label{ss:general-contact}
 
  The simple geometric model considered in this article ensures that we have globally defined vector fields $(X_\perp,V)$ generating the contact structure and that $[X_\perp,X]= KV$ and $[V,X]=-X_\perp$. It makes some aspects of the exposition somewhat lighter (e.g. regarding the normal form procedure) but it is not essential in our analysis. In fact, one would only need to have two locally defined generating vector fields $(X_2,X_3)$ on a $3$-dimensional manifold $\ml{N}$ so that the operator $\Delta_{\text{sR}}$ writes down locally as $X_2^*X_2+X_3^*X_3$ (modulo some order $0$ operator) where the adjoint is taken with respect to a smooth (nonvanishing) volume form and where one has (locally)
\begin{equation}\label{e:contact}T_q\ml{N}=\text{Span}(X_2(q),X_3(q),[X_2,X_3](q)).\end{equation}
This last condition ensures that $D=\text{Span}(X_2,X_3)$ is non-integrable and thus a contact structure. The H\"ormander type condition~\eqref{e:contact} is in fact the only ingredient needed to perform our normal form procedure in Section~\ref{s:normalform}. For the sake of exposition and as geodesic vector fields form a natural and rich family of Reeb vector fields, we refrain from considering the most general case and we focus on the somehow simplest example of contact structure\footnote{Another nice class of examples would be given by the operator $X^2+X_\perp^2$ on negatively/positively curved surfaces.} that is not already in normal form. In fact, we emphasize that, contrary to the flat Heisenberg case~\cite{ColindeVerdiereHillairetTrelat15, FermanianFischer21, FermanianLetrouit20}, the brackets $[X_\perp,X]$ and $[V,X]$ do not identically vanish. This implies that we do not have a nice algebraic structure at hand and that we have to take them into account in our analysis as it is the case in the general contact set-up treated in~\cite{ColindeVerdiereHillairetTrelat15}. In fact, as we shall see below, the way we deal with the normal form procedure slighlty differs from the one in~\cite{ColindeVerdiereHillairetTrelat15} by avoiding an ``explicit'' construction of symplectic coordinates and thus the use of Fourier integral operators. Yet this simplified method does not rely on the specific form of our operator. It would work as well for more general sub-Riemannian contact Laplacians in dimension $3$ modulo dealing with more cumbersome cohomological equations and modifying conveniently the various functions and vector fields in the subelliptic direction. In particular, if we write $-ihX_2$ and $-ihX_3$ as $\Op_h(H_2)$ and $\Op_h(H_3)$ (modulo terms of order $0$), then we could set $H_1:=\{H_3,H_2\}$ to measure the escape of mass at infinity. When studying the measure $\overline{\nu}_\infty$ (i.e. the regime $1\ll |H_1|\ll h^{-1}$), the geodesic vector field $X$ would be replaced as in~\cite{ColindeVerdiereHillairetTrelat15} by the Reeb vector field $X_1+\alpha X_2+\beta X_3$ with 
$$
X_1:=[X_3,X_2],\quad [X_2,X_1]=\beta X_1\ \text{mod}\ D,\quad \text{and}\quad [X_1,X_3]=\alpha X_1\ \text{mod}\ D.
$$
Except in the case $W=Q\equiv 0$, the invariance properties of the measures $\nu_{k,\infty}^\pm$ would however be more complicated (involving $\alpha$, $\beta$, $Q$, $k$ and the three vector fields $X_j$) and we do not try to compute them explicitely.

\subsection{Organization of the article}   
In Section~\ref{s:expression}, we fix the geometric and semiclassical conventions that are used all along the article. Then, in Section~\ref{s:cutoff}, we microlocalize our eigenfunctions in the region $|H_1|\gg 1$ and we introduce microlocal lifts of our measures that capture the escape of mass at infinity. This is where we introduce the new variable $E=hH_1$. In Section~\ref{s:support}, we show that these microlocal lifts are concentrated on certain quantized layers along the $E$-variable. In Section~\ref{s:normalform}, we introduce a simple normal form procedure that is well adapted to the geometry of our problem and we implement it in Section~\ref{s:invariance} to derive the invariance properties of our lifted measure. Section~\ref{s:main_results_and_proofs} summarizes the main results of the article and show how they can be used to derive Theorem~\ref{t:simplified_theorem} from the introduction. Section~\ref{s:torus} treats the simple example of the flat torus in view of illustrating our analysis in a concrete example. Finally, the article contains two appendices: one devoted to the spectral properties of our hypoelliptic operators (Appendix~\ref{a:spectral}) and another one collecting a few standard facts from semiclassical analysis (Appendix~\ref{a:pdo}). 

\subsection*{Acknowledgements} We thank G.~Carron, B.~Chantraine, C.~Fermanian-Kammerer, B.~Helffer, L.~Hillairet, F.~Maci\`a and C.~Sun for discussions related to various aspects of this work. Both authors are supported by the Agence Nationale de la Recherche through the PRC grant ADYCT (ANR-20-CE40-0017), the first author is partially supported by the projects MTM2017-85934-C3-3-P and PID2021-124195NB-C31 (MINECO,
Spain) and the second author acknowledges the support of the Institut Universitaire de France.

\section{Semiclassical conventions}\label{s:expression}

In this section we introduce the conventions from differential geometry and semiclassical analysis required for our study. We first recall the existence of local isothermal coordinates to write down the differential objects appearing in our framework. This will be particularly useful (although not essential) to describe the normal form procedure of Section \ref{s:normalform}. This local framework also allows us to define the necessary semiclassical pseudodifferential calculus, working from a fixed chart of $M$.

\subsection{Local isothermal coordinates}
\label{s:local_coordinates}

Near any given point $m_0\in M$, one can find a system of local coordinates $(x,y)\in U_0\subset \IR^2$ (with $(0,0)$ being the image of $m_0$) such that the metric $g$ writes down in a conformal way~\cite[Th.~3.4.8]{PaternainSaloUhlmann22}:
$$g(x,y):=e^{2\lambda(x,y)}\left(dx^2+dy^2\right).$$
We denote this neighborhood inside $M$ by $U$ in the sequel.

To write down the geometrical objects involved in the problem in terms of local isothermal coordinates, we follow the presentation of~\cite[\S 3.5]{PaternainSaloUhlmann22} and we refer to it for more details. 
If we denote by $z$ the angle between a unit vector $p\in S_qU_0$ and $\frac{\partial}{\partial x}$, then we have the following expressions for our vector fields~\cite[Lemma~3.5.6]{PaternainSaloUhlmann22}:
\begin{equation}\label{e:geod-coord}
 X:=e^{-\lambda}(\cos z\partial_x+\sin z\partial_y)+e^{-\lambda}\left(-\partial_x\lambda\sin z +\partial_y\lambda \cos z\right)\partial_z,
\end{equation}
\begin{equation}\label{e:hori-coord}
 X_\perp:=e^{-\lambda}(\sin z\partial_x-\cos z\partial_y)+e^{-\lambda}\left(\partial_x\lambda \cos z+\partial_y\lambda \sin z\right) \partial_z,
\end{equation}
and 
\begin{equation}\label{e:vert-coord}
 V:= \partial_z.
\end{equation}
These expressions are obtained by solving the Hamilton-Jacobi equation for the Hamiltonian function $e^{2\lambda(x,y)}(\xi^2+\eta^2)$ on $T^*\IR^2$. The Sasaki metric is not conformal in the system of coordinates $(x,y,z)$. Yet, the volume form has a simple expression
\begin{equation}\label{e:measure-local-coor}d\mu_{L}(x,y,z)=e^{2\lambda(x,y)}dxdydz.
\end{equation}

\begin{rema}
As the results we are aiming at are local, it will be sufficient to work in such a local chart. We use this chart for the sake of concreteness and for simplicity of exposition. Yet, our dynamical and semiclassical arguments would work as well for more general contact flows for which such a nice chart does not exist. 
\end{rema}

\begin{rema} Without loss of generality, we can extend these operators on $\mathcal{U}_0:=SU_0=U_0\times\mathbb{S}^1$ to operators on $\IR^2\times\IS^1$ by extending $\lambda$ into a smooth compactly supported function on $\IR^2$. 
\end{rema}


\subsection{Hamiltonian formulation}
\label{s:hamiltonian_formalism}

In the following, we will make use of different tools of semiclassical pseudodifferential calculus. This leads to define the symbols corresponding to the operators of interest. To this aim, we introduce the Hamiltonian functions associated with the orthonormal frame $(X,X_\perp, V)$. Namely, we define the following symbols on $T^*(\IR^2\times\IS^1)$:
\begin{equation}\label{e:H1}
 H_1(x,y,z,\xi,\eta,\zeta):=e^{-\lambda(x,y)}\left(\xi\cos z+\eta\sin z+\zeta \left(-\partial_x\lambda\sin z +\partial_y\lambda \cos z\right)\right),
\end{equation}
\begin{equation}\label{e:H2}
  H_2(x,y,z,\xi,\eta,\zeta):=e^{-\lambda(x,y)}\left(\xi\sin z-\eta\cos z+\zeta \left(\partial_x\lambda\cos z +\partial_y\lambda \sin z\right)\right),
\end{equation}
and
\begin{equation}\label{e:H3}
 H_3(x,y,z,\xi,\eta,\zeta):=\zeta.
\end{equation}

Notice, in particular, that there exists a positive constant $C_0$ (depending only on our local isothermal coordinates and on our extension of $\lambda$ to $\IR^2$) verifying
\begin{equation}\label{e-norm-comparison}C_0^{-1}(\xi^2+\eta^2+\zeta^2)\leq H_1^2+H_2^2+H_3^2\leq C_0 (\xi^2+\eta^2+\zeta^2).\end{equation}

The commutator relations~\eqref{e:commutators} can then be translated into the following Poisson bracket commutation formulas:
\begin{equation}\label{e:commutators-hamiltonian}
 \{H_1,H_3\}=H_2,\quad \{H_1,H_2\}=-KH_3,\quad\text{and}\quad\{H_2,H_3\}=-H_1,
\end{equation}
where we recall that $K(x,y)$ is the scalar curvature. We also collect a few useful relations involving $H_1$, $H_2$, $H_3$, and $\lambda$ in the next lemma, whose proof is immediate:

\begin{lemm} The following identities hold:
\begin{align}
\label{e:trivial-relation-H3}\partial_xH_3 & =\partial_yH_3=\partial_zH_3=0, \\[0.2cm]
\label{e:intertwine-H1H2}\partial_zH_2 & =H_1,\quad \partial_zH_1=-H_2, \\[0.2cm]
\label{e:derivative-H1}\left(\begin{array}{c} \partial_xH_1\\ \partial_yH_1\end{array}\right) & =-\left(\begin{array}{cc} \partial_x\lambda &e^{-\lambda}\left(\partial_x^2\lambda \cos z+\partial^2_{xy}\lambda\sin z\right)  \\ \partial_y\lambda &e^{-\lambda}\left(\partial_{xy}^2\lambda \cos z+\partial^2_{y}\lambda\sin z\right) \end{array}\right)\left(\begin{array}{c} H_1\\H_3\end{array}\right), \\[0.2cm]
\label{e:derivative-H2}\left(\begin{array}{c} \partial_xH_2\\ \partial_yH_2\end{array}\right) & =-\left(\begin{array}{cc} \partial_x\lambda &e^{-\lambda}\left(\partial_x^2\lambda \sin z-\partial^2_{xy}\lambda\cos z\right)  \\ \partial_y\lambda &e^{ -\lambda}\left(\partial_{xy}^2\lambda \sin z-\partial^2_{y}\lambda\cos z\right)  \end{array}\right)\left(\begin{array}{c} H_2\\H_3\end{array}\right).
\end{align}
\end{lemm}

\subsection{Semiclassical Weyl quantization}

With the above conventions, we can next rewrite the geometrical objects introduced in \S\ref{s:local_coordinates} in terms of pseudodifferential operators by making use of the Hamiltonian formulation of \S\ref{s:hamiltonian_formalism}. Precisely, we have:
$$\frac{h}{i}X=\Op_h^w(H_1-ih X(\lambda)),\ \frac{h}{i}X_\perp=\Op_h^w(H_2-ihX_\perp(\lambda)),\ \text{and}\ \frac{h}{i}V=\Op_h^w(H_3),$$
where $\Op_h^w$ stands for the semiclassical Weyl quantization on $\IR^2\times\IS^1$ (see Appendix~\ref{a:pdo}). In particular,
\begin{equation}\label{e:quantization-Delta-sR}
 -h^2\Delta_{\text{sR}}=\Op^w_h\left(H_2^2+H_3^2-2ih X_\perp(\lambda)H_2-h^2X_\perp(\lambda)^2\right).
\end{equation}
It will be slightly more convenient to work in the local chart with the operator $-h^2e^{\lambda}\Delta_{\text{sR}}e^{-\lambda}$ due to the following conjugation formula:
\begin{lemm}
The following holds on $\mathcal{U}_0$:
\begin{equation}\label{e:conjugation-sublaplacian}
 -h^2e^{\lambda}\Delta_{\operatorname{sR}}e^{-\lambda}=\Op_h^w(H_2^2+H_3^2).
\end{equation}
\end{lemm}

\begin{proof}
Let us define the modified sub-Riemannian Laplacian given by:
$$\tilde{\Delta}_{\text{sR}}:=X_\perp^2+V^2-2X_\perp(\lambda)X_\perp.$$ 
One can then verify that $e^{-\lambda} \tilde{\Delta}_{\text{sR}}e^{\lambda}=\Delta_{\text{sR}}+(X_\perp^2(\lambda)-X_\perp(\lambda)^2),$
and moreover that
$$-h^2\tilde{\Delta}_{\text{sR}}=\Op_h^w(H_2^2+H_3^2)-h^2\left(X_\perp^2(\lambda)-X_\perp(\lambda)^2\right).$$
In particular, one finds \eqref{e:conjugation-sublaplacian}.

\end{proof}
We are now in position to obtain the expression of the full operator $\widehat{P}_h$. To do that, we first use the Weyl pseudodifferential calculus to write
$$-ih^2QX -\frac{ih^2X(Q)}{2}=\Op_h^w(hQH_1+h^2W_1),$$
where $W_1\in\ml{C}^{\infty}_c(\IR^2\times\IS^1,\IC)$ is independent of $h$. Using the composition rules for the Weyl quantization, we obtain
\begin{equation}\label{e:conjugation-Ph}
\widehat{P}_h= e^{-\lambda}\Op_h^w(H_2^2+H_3^2+hQH_1+W+h^2W_{1,\lambda}) e^{\lambda},
\end{equation}
where $W_{1,\lambda}\in\ml{C}^{\infty}_c(\ml{U}_0,\IC)$ is independent of $h$.  Regarding this expression, it is natural to set
\begin{equation}\label{e:conjugated-eigenfunction}
u_h=e^{\lambda}\psi_h,
\end{equation}
and thanks to~\eqref{e:conjugation-Ph}, $u_h$ solves locally in $\ml{U}_0$ the eigenvalue equation
\begin{equation}\label{e:eigenvalue-conjugation}
 \widehat{P}_{h,\lambda} \, u_h=\lambda_h \, u_h,
\end{equation}
where
$$\widehat{P}_{h,\lambda}:=\Op_h^w(H_2^2+H_3^2+hQH_1+W+h^2W_{1,\lambda}).$$
The a priori estimate~\eqref{e:apriori-semiclassical} from the introduction then reads
\begin{equation}\label{e:apriori-semiclassical1}
\begin{array}{rcl} \displaystyle \|\Op_h^w(H_2)u_h\|_{L^2(\ml{K})}+\|\Op_h^w(H_3)u_h\|_{L^2(\ml{K})}+\|\Op_h^w(hH_1)u_h\|_{L^2(\ml{K})} & & \\[0.2cm]
 &   \hspace*{-6.5cm}  \leq &  \hspace*{-3.3cm}  \displaystyle 2C_{Q,W,\mathcal{K}}(1+|\lambda_0|),
  \end{array}
\end{equation}
where $\mathcal{K}$ is any compact subset of $\ml{U}_0$ and the $L^2$ norm is now taken with respect to the standard Lebesgue measure $dxdydz$ on $\mathcal{K}\subset \IR^2\times\IS^1$.

\subsection{Class of symbols in the region $|H_1|\gg\sqrt{H_2^2+H_3^2}$} In Section \ref{s:normalform} below, we will describe a normal form procedure that will naturally involve functions in the spaces:
$$\mathcal{P}_N(\IR^2\times\IS^1):=\Bigg \{\sum_{\alpha=(\alpha_2,\alpha_3)\in\IZ_+^2 \, : \, |\alpha|\leq N } a_{\alpha}\left(\frac{H_2}{H_1}\right)^{\alpha_2}\left(\frac{H_3}{H_1}\right)^{\alpha_3} \, : \,\forall\alpha,\ a_\alpha\in\ml{C}^{\infty}(\IR^2\times\IS^1)\Bigg \}.$$
Notice that as a consequence of \eqref{e:trivial-relation-H3}, \eqref{e:intertwine-H1H2}, \eqref{e:derivative-H1} and \eqref{e:derivative-H2}, one can verify the following:
\begin{lemm}\label{l:polynomial} Let $P$ be an element in $\ml{P}(\IR^2\times\IS^1) := \bigcup_N \mathcal{P}_N(\IR^2\times\IS^1)$. Then, $\partial_xP,$ $ \partial_yP$ and $\partial_z P$
 belong to $\ml{P}(\IR^2\times\IS^1)$. Similarly, letting $p=(\xi,\eta,\zeta)$, one has that, for every $\gamma\in\IZ_+^3$,
 $H_1^{|\gamma|}\partial_{p}^{\gamma}P$
 belongs to $\ml{P}(\IR^2\times\IS^1).$
\end{lemm}
\begin{proof}
 The first part of the Lemma is direct consequence of~\eqref{e:trivial-relation-H3}, \eqref{e:intertwine-H1H2}, \eqref{e:derivative-H1} and \eqref{e:derivative-H2}. For the second part, we proceed by induction on $|\gamma|$ and use the fact that $H_1$, $H_2$ and $H_3$ are linear functions in $(\xi,\eta,\zeta)$.
\end{proof}

To make the necessary estimates in the sub-elliptic regime arising from our problem, we will be naturally led to work in the ``conic'' region
\begin{equation}\label{e:cone-eps}
C_\varepsilon(\mathcal{K}):=\left\{(q,p)\in T^*(\mathcal{K}):\ \varepsilon |H_1(q,p)|\geq \sqrt{1+H_2^2(q,p)+H_3^2(q,p)}\right\},
\end{equation}
where $\mathcal{K}$ is a compact subset of $\IR^2\times\IS^1$ and where $\varepsilon\in(0,1]$ is some small parameter that is intended to tend to $0$ in the end. We record the following corollary of Lemma~\ref{l:polynomial}: 
\begin{coro}\label{c:derivative-cone} Let $P$ in $\ml{P}(\IR^2\times\IS^1)$ and let $\ml{K}$ be a compact subset of $\IR^2\times\IS^1$. For every $0<\varepsilon\leq 1$ and for every $(\alpha,\beta)\in\IZ_+^6$, one can find a constant $C_{\varepsilon, P, \ml{K}, \alpha,\beta}$ such that, for every $(q,p)=(x,y,z,\xi,\eta,\zeta)$ in $C_\varepsilon(\ml{K})$, one has
$$\left|\partial^\alpha_{q}\partial_{p}^\beta P\right|\leq C_{\varepsilon,P, \ml{K},\alpha,\beta}\langle p\rangle^{-|\beta|},$$ 
where $\langle p\rangle:=(1+\xi^2+\eta^2+\zeta^2)^{\frac{1}{2}}.$ 
\end{coro}
In particular, elements in $\ml{P}(\IR^2\times\IS^2)$ satisfy the properties of the class of (Kohn-Nirenberg) symbols $S^0_{\text{cl}}(T^*(\IR^2\times\IS^1))$ defined in Appendix~\ref{a:pdo} inside $C_\varepsilon(\mathcal{K})$ for any compact subset $\mathcal{K}$ of $\IR^2\times\IS^1$.

\section{Reduction to the region $1\ll |H_1|\lesssim h^{-1}$}\label{s:cutoff}

Since our results on quantum limits and semiclassical measures in the subelliptic regime are essentially local, we will restrict ourselves to study the following measures on $\mathcal{U}_0$:
$$\nu_h:a\in\ml{C}^\infty_c(\mathcal{U}_0)\mapsto \int_{\mathcal{U}_0} a(x,y,z)|\psi_h(x,y,z)|^2e^{2\lambda(x,y)}dxdydz,$$
where $\mathcal{U}_0$ is a bounded open subset of $\IR^2\times\IS^1$ given by local isothermal coordinates and where $(\psi_h,\lambda_h)$ is a sequence satisfying~\eqref{e:semiclassical-eigenvalue}. As the sequence $(\psi_h)$ is normalized, this defines a sequence of measures on $\mathcal{U}_0$ that are of finite mass $\leq 1$. In fact, up to an extraction, one can suppose that $\nu_h\rightharpoonup\nu$ as $h\rightarrow 0^+$ and the limit measure is supported in $\ml{U}_0$ with total mass $\leq 1$. We fix this converging subsequence for the rest of the article.

Using the convention from~\eqref{e:eigenvalue-conjugation}, this can be rewritten as
$$\nu_h:a\in\ml{C}^\infty_c(\mathcal{U}_0)\mapsto \int_{\mathcal{U}_0} a(x,y,z)|u_h(x,y,z)|^2dxdydz,$$
which allows us to work with the standard Lebesgue measure. Before trying to prove our main Theorem, we will first show in this Section how to reduce these integrals to the region of the phase space where $1\ll |H_1|\lesssim h^{-1}$ and thus how to define the mesure $\mu_\infty$ from the introduction. We remark that the measures $\nu_h$ can be rewritten as
\begin{equation}\label{e:measure-pseudo}
\langle\nu_h,a\rangle=\left\langle \Op_h^w(a)u_h,u_h\right\rangle_{L^2}.
\end{equation}
More generally, as anticipated in the introduction, we consider the associated Wigner distribution 
\begin{equation}\label{e:measure-pseudo_lift}
\forall a\in\ml{C}^{\infty}_c(T^*\ml{U}_0),\quad\langle w_h,a\rangle: =\left\langle \Op_h^w(a)u_h,u_h\right\rangle_{L^2}, 
\end{equation}
and, up to another extraction, we can suppose that it converges to some (finite) limit measure $w$ on $T^*\ml{U}_0$.

\begin{rema} As usual when working with coordinate charts, we make a small abuse of notations and write $\psi_h$ for the image of $\psi_h$ in the coordinate system. As we always suppose $a$ to be compactly supported in the chart, this causes no difficulties (up to $\ml{O}(h^\infty)$ remainders) and we may view $\psi_h$ as a smooth compactly supported function on $\IR^2\times\IS^1$.
\end{rema}

In the rest of the article, we fix a smooth function $\chi:\IR\rightarrow [0,1]$ which is equal to $1$ on $[-1,1]$ and to $0$ outside $[-2,2]$. Moreover, we make the assumption that $\chi'\geq 0$ on $\IR_-$ and $\chi'\leq 0$ on $\IR_+$. For such a function, we also set
\begin{equation}\label{e:partition-unity}
 \tilde{\chi}=1-\chi.
\end{equation}

For the description of the limit measure $\nu_\infty = \nu - \pi_*w$ and for the definition of $\mu_\infty$, we reduce the analysis of the sequence $\nu_h$ to the subelliptic regime $1 \ll \vert H_1 \vert \lesssim h^{-1}$. To do so, we proceed in three steps.

\subsection{Reduction to the region at infinity}

First, we split the measure $\nu_h$ into two parts corresponding to the compact and non-compact distribution of the sequence $(u_h)_{h\rightarrow 0^+}$ in phase space. It leads respectively to the definition of the weak limits $\pi_* w$ and $\nu_\infty$. Let $R>1$, we introduce the  cut-off functions
\begin{equation}\label{e:cutoff-cone}
\chi_R^B:=\chi\left(\frac{H_1^2+H_2^2+H_3^2}{R}\right),\ \quad\tilde{\chi}_R^B=1-\chi_R^B.
\end{equation}
These cut-offs allow us to split $\nu_h=\nu_{h,R}+\nu_h^R$ where
$$\forall a\in\ml{C}^\infty_c(\ml{U}_0),\quad\langle \nu_{h,R},a\rangle=\langle w_h, a\chi_R^B\rangle,\quad\text{and}\quad\langle \nu_{h}^R,a\rangle=\langle w_h, a\tilde{\chi}_R^B\rangle.$$
Notice moreover that the cut-offs $\chi_R^B$ and $\tilde{\chi}_R^B$ belong to the admissible class of symbols $S^0_{\operatorname{cl}}(T^*(\IR^2\times\IS^1)$ defined in Appendix~\ref{a:pdo}. Letting $h\rightarrow 0^+$ and $R\rightarrow +\infty$ (in this order), one finds
$$\lim_{R\rightarrow +\infty}\lim_{h\rightarrow 0^+}\langle \nu_{h,R},a\rangle= \langle\pi_* w,a\rangle,$$
where $\pi: T^*\ml{U}_0 \ni (q,p) \mapsto q\in\ml{U}_0$, and
\begin{equation}\label{e:QL-infinity}
 \langle \nu_\infty,a\rangle = \lim_{R\rightarrow +\infty}\lim_{h\rightarrow 0^+}\langle \nu_{h}^R,a\rangle.
\end{equation}

\begin{rema}
 Again we implicitely consider sequences $R_n\rightarrow +\infty$ (say $2^n$) but we just write $R\rightarrow +\infty$ for simplicity.
\end{rema}

\subsection{Reduction to the cones $C_\varepsilon(\overline{\ml{U}}_0)$} We next introduce a further cut-off restricting the measure $\nu_h$ to a conic region containing the semiclassical wave-front set of the sequence $(u_h)_{h\rightarrow 0^+}$ in the subelliptic regime $1 \ll \vert H_1 \vert\lesssim h^{-1}$. We set, for $0<\varepsilon< 1$,
\begin{equation}\label{e:cutoff-cone_2}
\chi_\varepsilon^C:=\chi\left(\frac{\varepsilon H_1}{\sqrt{H_2^2+H_3^2+1}}\right),\ \quad\tilde{\chi}_\varepsilon^C=1-\chi_\varepsilon^C.
\end{equation}
Before including these cut-offs in our analysis of the sequence $\nu_h^R$, we show the following result:
\begin{lemm}\label{l:symbol} For every $0<\varepsilon<1$, the symbols $\chi_\varepsilon^C$ and $\tilde{\chi}_\varepsilon^C$ belong to the admissible class of symbols $S^0_{\operatorname{cl}}(T^*(\IR^2\times\IS^1)$ defined in Appendix~\ref{a:pdo}. 
\end{lemm}
\begin{rema}
The corresponding seminorms of $\chi_\varepsilon^C$ and $\tilde{\chi}_\varepsilon^C$ in $S^0_{\operatorname{cl}}(T^*(\IR^2\times\IS^1)$ depend on $\varepsilon$.
\end{rema}
\begin{proof}
 From the definition of $\chi_\varepsilon^C$ and $\tilde{\chi}_\varepsilon^C$, it is sufficient to verify that all the derivatives of
$$g:=\frac{H_1}{\sqrt{H_2^2+H_3^2+1}}$$
are bounded (with some further decay for the derivatives with respect to $(\xi,\eta,\zeta)$) in the region where $\varepsilon g\in\text{supp}(\chi')$, i.e.
\begin{equation}\label{e:support-derivatives}
 \frac{1}{\varepsilon}\sqrt{1+H_2^2+H_3^2}\leq |H_1|\leq \frac{2}{\varepsilon}\sqrt{1+H_2^2+H_3^2}.
\end{equation}
Thanks to~\eqref{e:trivial-relation-H3}, \eqref{e:intertwine-H1H2}, \eqref{e:derivative-H1} and \eqref{e:derivative-H2}, we can verify by induction that, for every $\alpha\in\IZ_+^3$,
$$\partial_{xyz}^{\alpha}g=\sum_{j=0}^{|\alpha|}\frac{P_{j,\alpha}(H_1,H_2, H_3)}{(1+H_2^2+H_3^2)^{j+\frac{1}{2}}},$$
where, for every $0\leq j\leq |\alpha|$, $(u,v)\mapsto P_{j,\alpha}(u,v)$ is a polynomial of degree $\leq 2j+1$. In particular, using~\eqref{e:support-derivatives}, all these derivatives are bounded (with a constant depending on $\varepsilon$). It now remains to deal with the derivatives with respect to $(\xi,\eta,\zeta)$. To this aim, we recall that $H_1$, $H_2$ and $H_3$ are linear functions in these variables. Arguing by induction as for the derivatives with respect to $(x,y,z)$, we can then conclude that, for every $(\alpha,\beta)\in \IZ_+^6$, $\partial_{xyz}^\alpha\partial_{\xi\eta\zeta}^\beta g$ is uniformly bounded by $C_\varepsilon (1+H_2^2+H_3^2)^{-\frac{|\beta|}{2}}$ under the assumption~\eqref{e:support-derivatives}. Using the upper bound in~\eqref{e:support-derivatives}, we deduce that this is bounded by $C_\varepsilon(1+H_1^2+H_2^2+H_3^2)^{-\frac{|\alpha|}{2}}$, for some slighlty larger constant $C_\varepsilon>0$. This concludes the proof of the Lemma thanks to~\eqref{e-norm-comparison}.
\end{proof}

The next step of our analysis consists of inserting these two cutoffs in the construction of $\nu_h^R$. This produces the splitting:
$$\langle\nu_h^R,a\rangle=\left\langle \Op_h^w(a\tilde{\chi}_R^B\chi_\varepsilon^C)u_h,u_h\right\rangle_{L^2}
+\left\langle \Op_h^w(a\tilde{\chi}_R^B\tilde{\chi}_\varepsilon^C)u_h,u_h\right\rangle_{L^2}.
$$
By the same arguments as the ones used in the proof of Lemma~\ref{l:symbol}, we observe that the function 
$$
\frac{a\chi_\varepsilon^C\tilde{\chi}_R^B}{1+H_2^2+H_3^2}
$$ 
belongs to the class of symbols $S^{-2}_{\text{cl}}(T^*(\IR^2\times\IS^1))$. In particular, from the composition rules for pseudodifferential operators and the Calder\'on-Vaillancourt Theorem, one has
$$\Op_h^w(a\chi_\varepsilon^C\tilde{\chi}_R^B)=\Op_h^w\left(\frac{a\chi_\varepsilon^C\tilde{\chi}_R^B}{1+H_2^2+H_3^2}\right)\Op_h^w\left(1+H_2^2+H_3^2\right)+\mathcal{O}_{L^2\rightarrow L^2}(h).$$
Combining this composition rule with~\eqref{e:eigenvalue-conjugation} and~\eqref{e:apriori-semiclassical1}, one obtains the estimate
\begin{equation}\label{e:measure-pseudo2}
\left|\left\langle \Op_h^w(a\tilde{\chi}_R^B\chi_\varepsilon^C)u_h,u_h\right\rangle_{L^2}\right|\leq \left\|\Op_h^w\left(\frac{a\chi_\varepsilon^C\tilde{\chi}_R^B}{1+H_2^2+H_3^2}\right)\right\|_{L^2\rightarrow L^2}+\ml{O}_{\varepsilon,R}(h). 
\end{equation}
Moreover, by construction of our cutoff functions and by using Calder\'on-Vaillancourt Theorem one more time, one gets
$$\left|\left\langle \Op_h^w(a\tilde{\chi}_R^B\chi_\varepsilon^C)u_h,u_h\right\rangle_{L^2}\right|\leq \frac{C_{M,g,a}}{R\varepsilon^2}+\ml{O}_{\varepsilon,R}(h^{\frac{1}{2}}).$$
Hence, one ends up with
\begin{equation}\label{e:measure-with-cutoffs}\langle\nu_h^R,a\rangle=\left\langle \Op_h^w(a\tilde{\chi}_R^B\tilde{\chi}_\varepsilon^C)u_h,u_h\right\rangle_{L^2}+\ml{O}((R\varepsilon^2)^{-1})+\ml{O}_{\varepsilon,R}(h^{\frac{1}{2}}),
\end{equation}
and we can introduce the object of interest for our analysis:
\begin{equation}\label{e:measure-cone-infinity}
\langle\nu_h^{R,\varepsilon},a\rangle:=\left\langle \Op_h^w(a\tilde{\chi}_R^B\tilde{\chi}_\varepsilon^C)u_h,u_h\right\rangle_{L^2}.
\end{equation}
In conclusion, we have shown:
\begin{lemm}\label{l:limit-measure} With the above conventions, one has, for every $0<\varepsilon<1$,
 $$
 \langle \nu_\infty,a\rangle =\lim_{R\rightarrow +\infty}\lim_{h\rightarrow 0^+}\langle\nu_h^{R,\varepsilon},a\rangle.
$$
 \end{lemm}

\begin{rema}\label{r:role-epsilon}
Note that, so far, the parameter $\varepsilon>0$ does not play any particular role. However, it will become important when analyzing the invariance and support properties of $\nu_\infty$ where we will need to take the limit $\varepsilon \to 0$. 
\end{rema}

Finally, we record the following useful Lemma that follows from the proof of Lemma~\ref{l:symbol}:
\begin{lemm}\label{l:symbol-cone} Let $\mathcal{K}$ be a compact subset of $\IR^2\times\IS^1$. For every $0<\varepsilon\leq 1$, for every $N_0\geq 2$, and for every $(\alpha,\beta)\in\IZ_+^6$, one can find a positive constant $C_{\varepsilon, N_0, \mathcal{K}, \alpha,\beta}$ such that, for every $(q,p)=(x,y,z,\xi,\eta,\zeta)$ in $C_{2^{N_0}\varepsilon}(\mathcal{K})\setminus C_{2^{-N_0}\varepsilon}(\mathcal{K})$, one has
$$\left|\partial^\alpha_{q}\partial_{p}^\beta \left(\frac{H_1}{\sqrt{1+H_2^2+H_3^2}}\right)\right|\leq C_{\varepsilon, N_0, \mathcal{K},\alpha,\beta}\langle p\rangle^{-|\beta|},$$ 
where $\langle p\rangle:=(1+\xi^2+\eta^2+\zeta^2)^{\frac{1}{2}}.$ 
\end{lemm}
In other words, the function $\frac{H_1}{\sqrt{1+H_2^2+H_3^2}}$ belongs to an amenable class of symbols inside the ``cone'' $C_{2^{N_0}\varepsilon}(\mathcal{K})\setminus C_{2^{-N_0}\varepsilon}(\mathcal{K})$.

\subsection{Reduction to the region $ 1 \ll \vert H_1 \vert \lesssim h^{-1}$} 

We will now localize the phase-space distribution of the sequence $(u_h)$ in the sub-elliptic region $ 1 \ll \vert H_1 \vert \lesssim h^{-1}$. To do that, we introduce the cutoff functions, for $R_1>1$,
$$\rho_{R_1}:=\tilde{\chi}\left(\frac{hH_1}{R_1}\right),\ \text{and}\ \tilde{\rho}_{R_1}:=1-\rho_{R_1}.$$
The cut-off $\tilde{\rho}_{R_1}$ localizes the sequence $(u_h)$ in the region of interest to us. Let us first show that this last localization keeps the analysis in the admissible symbol class.

\begin{lemm}\label{l:cutoffH1} Let $a\in\ml{C}^{\infty}_c(\ml{U}_0)$. The functions $a\tilde{\chi}_R^B\tilde{\chi}_\varepsilon^C\rho_{R_1}$ and $a\tilde{\chi}_R^B\tilde{\chi}_\varepsilon^C\tilde{\rho}_{R_1}$ belong to the admissible symbol class $S_{\operatorname{cl}}^0(T^*(\IR^2\times\IS^1))$ with seminorms that are uniformly bounded for $0<h\leq 1$.
\end{lemm}
\begin{proof}
 We already know from Lemma~\ref{l:symbol} that $a\tilde{\chi}_\varepsilon^C$ belongs to $S_{\operatorname{cl}}^0(T^*(\IR^2\times\IS^1))$ and we have observed that $\tilde{\chi}_R^B$ belongs to $S_{\operatorname{cl}}^0(T^*(\IR^2\times\IS^1))$. Thus, we need to show that $\rho_{R_1}(hH_1)$ belongs to this class when restricted to the region of phase space given by the support of $a\tilde{\chi}_\varepsilon^C\tilde{\chi}_R^B$. Among other constraints, in this region we have that $(x,y,z)$ in $\ml{U}_0$ and
 \begin{equation}\label{e:support-cone}\frac{\varepsilon |H_1|}{\sqrt{1+H_2^2+H_3^2}}\geq 1.
\end{equation}
In other words, we need to show that the derivatives of $hH_1$ verify the properties of the class $S^{0}_{\text{cl}}(\IR^2\times\IS^1)$ under these support properties and the additional assumption that $\rho_{R_1}'(hH_1)\neq 0$, which leads to the additional constraint
 \begin{equation}\label{e:support-H1}
 R_1\leq h|H_1|\leq 2R_1.
\end{equation}
In view of~\eqref{e:trivial-relation-H3}, \eqref{e:intertwine-H1H2}, \eqref{e:derivative-H1} and \eqref{e:derivative-H2}, one can verify that all the derivatives of $\rho_{R_1}$ of order $l$ with respect $(x,y,z)$ are linear combinations of functions of the form
$$P(hH_1,hH_2,hH_3)\chi^{(k)}(hH_1/R_1),$$
where $P$ is a polynomial of degree $k$ with coefficients in $\mathcal{C}^{\infty}(\IR^2\times\IS^1)$ and with $k\leq l$. In particular, thanks to the support properties~\eqref{e:support-cone} and~\eqref{e:support-H1}, these quantities are bounded as expected. It now remains to differentiate these quantities with respect to $(\xi,\eta,\zeta)$. As $H_1$, $H_2$, $H_3$ are polynomials of degree $1$ in $(\xi,\eta,\zeta)$, it has the effect to lower the degree of the polynomial and to get a bound of order $h^{l'}$ where $l'$ is the number of derivatives with respect to these variables. Using~\eqref{e:support-cone} and~\eqref{e:support-H1} one more time together with~\eqref{e-norm-comparison}, this yields the expected decaying properties of the class $S_{\operatorname{cl}}^0(T^*(\IR^2\times\IS^1))$ with constants that are independent of $0<h\leq 1$.
\end{proof}
We next include these cutoff functions in~\eqref{e:measure-cone-infinity}. The goal is to verify that the contribution of the term
\begin{equation}\label{e:largeH1}
 \left\langle \Op_h^w\left(a\tilde{\chi}_\varepsilon^C\tilde{\chi}_R^B\tilde{\chi}\left(\frac{hH_1}{R_1}\right)\right)u_h,u_h\right\rangle_{L^2}
\end{equation}
is small as $R_1\rightarrow +\infty$. Notice, to this aim, that on the support of $a\tilde{\chi}_\varepsilon^C\tilde{\chi}_R^B$, the function $1/H_1$ belongs to the class of symbols $S^0_{\text{cl}}(T^*(\IR^2\times\IS^1))$. Hence, by the composition rules for pseudodifferential operators, one has
\begin{multline*}\left\langle \Op_h^w\left(a\tilde{\chi}_\varepsilon^C\tilde{\chi}_R^B\tilde{\chi}\left(\frac{hH_1}{R_1}\right)\right)u_h,u_h\right\rangle_{L^2}\\[0.2cm]
=\left\langle \Op_h^w\left(\frac{a\tilde{\chi}_\varepsilon^C\tilde{\chi}_R^B\tilde{\chi}\left(hH_1/R_1\right)}{hH_1}\right)\Op_h^w(hH_1)u_h,u_h\right\rangle_{L^2}+\ml{O}_{R,R_1,\varepsilon}(h).
\end{multline*}
Using then the a priori estimate~\eqref{e:apriori-semiclassical1} together with the Calder\'on-Vaillancourt Theorem, we find
$$\left\langle\Op_h^w\left(a\tilde{\chi}_\varepsilon^C\tilde{\chi}_R^B\tilde{\chi}\left(\frac{hH_1}{R_1}\right)\right)u_h,u_h\right\rangle_{L^2}
=\ml{O}(R_1^{-1})+\ml{O}_{R,R_1,\varepsilon}(h).
$$
Therefore, by another application of pseudodifferential calculus rules, we get finally that
\begin{equation}\label{e:measure-pseudo4}
\langle\nu_h^{R,\varepsilon},a\rangle=\left\langle \Op_h^w(a\tilde{\rho}_{R_1}(hH_1)\tilde{\chi}_\varepsilon^C\tilde{\chi}_R^B)u_h,u_h\right\rangle_{L^2(\ml{U}_0)}+\ml{O}(R_1^{-1})+\ml{O}_{R,R_1,\varepsilon}(h).
\end{equation}

\subsection{Adding a new variable $E=hH_1$}\label{s:new_variable}
To complete the preliminaries concerning the phase-space localization of the measure $\nu_h^{R,\varepsilon}$, we lift slightly our analysis by introducing more general distributions in terms of a new variable $h H_1$ for the symbols $a$ considered above. Namely, we set 
\begin{equation}\label{e:wigner-H1}\mu_{h}^{R,\varepsilon}: b\in\ml{C}_c^{\infty}(\ml{U}_0\times\IR)\mapsto\left\langle \Op_h^w(b(x,y,z,hH_1)\tilde{\chi}_R^B\tilde{\chi}_\varepsilon^C)u_h,u_h\right\rangle_{L^2(\ml{U}_0)}.\end{equation}
The same argument as in the proof of Lemma~\ref{l:cutoffH1} shows that $b(x,y,z,hH_1)\tilde{\chi}_\varepsilon^C\tilde{\chi}_R^B$ belongs to $S^0_{\text{cl}}(T^*(\IR^2\times\IS^1))$. In particular, this defines a bounded sequence in $\mathcal{D}^\prime(\ml{U}_0\times\IR)$. Thus, up to another extraction, we may suppose that $\mu_{h}^{R,\varepsilon}$ converges (for the weak-$\star$ topology) to some distribution $\mu^{R,\varepsilon}$. Thanks to the Garding inequality~\eqref{e:Garding0}, this is a positive distribution, thus a finite measure. Moreover, since the sequence $(\psi_h)$ is normalized (and hence $(u_h)$ is bounded), this defines a finite measure. Up to another extraction, we can suppose that $\mu^{R,\varepsilon}$ weakly converges to some limit measure $\mu^{\varepsilon}$ as $R\rightarrow +\infty$. Coming back to~\eqref{e:measure-pseudo4}, we find that
\begin{equation}\label{e:comparison-measure0}\langle\nu_\infty,a\rangle=\lim_{R\rightarrow+\infty}\lim_{h\rightarrow 0^+}\langle\nu_h^{R,\varepsilon},a\rangle=\langle \mu^{\varepsilon}, a\tilde{\rho}_{R_1} (E)\rangle+\ml{O}(R_1^{-1}).
 \end{equation}
Applying the dominated convergence Theorem, one finds that, for all $\varepsilon>0$ (small enough),
\begin{equation}\label{e:comparison-measure}\forall a\in\ml{C}^{\infty}_c(\ml{U}_0),\quad\langle\nu_\infty,a\rangle=\langle\mu^{\varepsilon},a\rangle,
\end{equation}
so that our analysis boils down to the description of the measure $\mu^{\varepsilon}$. Finally, up to another extraction as $\varepsilon\rightarrow 0^+$, we can suppose that $\mu^{\varepsilon}$ converges to some (finite) Radon measure on $\ml{U}_0\times\IR$, i.e.
\begin{equation}\label{e:full-wigner-H1}
 \forall b\in\ml{C}_c^{\infty}(\ml{U}_0\times\IR),\ \langle \mu_\infty,b\rangle=\lim_{\varepsilon\rightarrow 0^+}\lim_{R\rightarrow+\infty}\lim_{h\rightarrow 0^+}\langle\nu_h^{R,\varepsilon},b\rangle.
\end{equation}
From~\eqref{e:comparison-measure0}, one has 
\begin{equation}\label{e:disintegration}
\forall a\in\ml{C}^0_c(\ml{U}_0),\quad \int_{\ml{U}_0}a(q)d\nu_{\infty}(q)=\int_{\ml{U}_0\times\IR}a(q)d\mu_{\infty}(q,E),
\end{equation}
and our analysis thus boils down to the properties of the extended measure $\mu_\infty$. 

\begin{rema}
 As explained in Remark~\ref{r:role-epsilon}, the fact that we take $\varepsilon\rightarrow 0^+$ is not important so far but will turn to be later on.
\end{rema}

\section{Support of the limit measure}\label{s:support}

Before describing the propagation and invariance properties of $\nu_\infty$, we discuss first the support properties of $\mu_\infty$ along the new variable $E\in \IR$. More precisely, the goal of this section is to prove the following:

\begin{prop}
\label{p:support_properties}
The measure $\mu_\infty$ defined in~\eqref{e:full-wigner-H1} decomposes as
$$
\mu_\infty(q,E) = \overline{\mu}_\infty(q,E) + \sum_{k=0}^\infty \big( \mu_{k,\infty}^+(q,E) + \mu_{k,\infty}^-(q,E) \big),
$$
where $\overline{\mu}_\infty$ and $(\mu_{k,\infty}^\pm)_{k\geq 0}$ are finite non-negative Radon measures on $\ml{M}\times\IR$ satisfying the
following concentration properties:
\begin{enumerate}[(S.1)]

\item $\operatorname{supp} \overline{\mu}_\infty \subset \mathcal{M}_{\lambda_0,W} \times \{ 0 \}$;
\medskip

\item for every $k \in \mathbb{Z}_+$, $ \operatorname{supp} \mu_{k,\infty}^\pm \subset \mathcal{H}_\pm^{-1}(2k+1) \subset \mathcal{U}_{\lambda_0,W}\times\IR_{\pm}^*.$
\end{enumerate}
\end{prop}

Recall that $\ml{M}_{\lambda_0,W}:=\{\lambda_0-W\geq 0\}$ is the classical allowed region, $\mathcal{U}_{\lambda_0,W} = \{ \lambda_0 - W > 0 \}$, and that $\ml{H}_{\pm}$ was defined in~\eqref{e:preserved-quantity}. Even if we do not explicitly use it, the proof of this result relies implicitely on the fact that our operators are locally modeled on the $3$-dimensional Heisenberg group. In particular, the presence of the quantum levels $2k+1$ is just a manifestation that the Heisenberg group is associated with the one-dimensional harmonic oscillator.

\subsection{Preliminary lemmas}
\label{s:preliminary_section}

We first define the following creation and annihilation type operators (ladder operators):
$$A_h:=\Op_h^w(H_2+iH_3)\quad\text{and}\quad A_h^*:=\Op_h^w(H_2-iH_3),$$
so that
\begin{equation}\label{e:AAstar}
\Op_h^w(H_2^2+H_3^2)=A_h^*A_h+h\Op_h^w(H_1)+h^2c_0=A_hA_h^*-h\Op_h^w(H_1)+h^2c_0,
\end{equation}
where $c_0$ is a smooth compactly supported and real-valued function on $\IR^2\times\IS^1$ (recall that $\lambda\equiv 0$ outside a compact set containing the isothermal neighborhood $U_0$) which is independent of $h$. Notice that, by \eqref{e:commutators-hamiltonian},
\begin{equation}
\label{e:commutator_A_A*}
[A_h, A_h^*] = 2h \Op_h^w(H_1).
\end{equation} 
We begin with the following lemma:
\begin{lemm}\label{l:A:power-k} Let $k\geq 1$. Then
$$A_h^k=\Op_h^w\left((H_2+iH_3)^k+\sum_{j=1}^{k}P_{j,k,h}(hH_1,hH_2,hH_3)(H_2+iH_3)^{k-j}\right),$$
where $P_{j,k,h}(u,v,w)$ is a polynomial with coefficients depending polynomially on $h$ and smoothly on $(x,y,z)\in\IR^2\times\IS^1$ (with uniformly bounded derivatives). 
\end{lemm}
In fact, modulo some extra work, we could be slightly more precise on the nature of the polynomials as we know that the full symbol is a polynomial of degree $k$ in the cotangent variables $(\xi,\eta,\zeta)$. Yet, as it is not necessary for our analysis, we do not try to be more precise and we just keep track of the informations that are relevant for our proofs.
\begin{proof}
Recall that
$$A_h=\Op_h^w(H_2+iH_3)=\frac{h}{i}X_\perp+hV+ihX_\perp(\lambda).$$
In particular, $A_h^k$ is a differential operator of order $k$ for every $k\geq 1$, and its symbol is polynomial of degree $\leq k$ in the cotangent variables $(\xi,\eta,\zeta)$. We now proceed by induction and suppose that the lemma is true for a given $k\geq 1$. Using the composition rule from Theorem~\ref{t:composition}, we can write
\begin{align*}
A_h^{k+1} &  =\Op_h^w\left((H_2+iH_3)^k+\sum_{j=1}^{k}P_{j,k,h}(hH_1,hH_2,hH_3)(H_2+iH_3)^{k-j}\right)\Op_h^w(H_2+iH_3)\\[0.2cm]
&  =\Op_h^w\left((H_2+iH_3)^{k+1}+\sum_{j=1}^{k}P_{j,k,h}(hH_1,hH_2,hH_3)(H_2+iH_3)^{k+1-j}\right) \\[0.2cm]
& \quad +\Op_h^w\left(R_k(h)\right),
\end{align*} 
where 
$$R_k(h)=
\sum_{\ell=1}^{k+1}\sum_{j=1}^{k}\frac{h^{\ell}}{\ell!}A(D)^\ell\left(\left(P_{j,k,h}(hH_1,hH_2,hH_3)(H_2+iH_3)^{k-j}\right)(H_2+iH_3)\right).$$
Here the sum stops at $\ell=k+1$ as each symbol is a polynomial of respective degree $k$ and $1$ in the $(\xi,\eta,\zeta)$ variables. Recall that $A(D)=\frac{1}{2i}\left(\partial_{p_1}\cdot\partial_{q_2}-\partial_{q_1}\cdot\partial_{p_2}\right)$ so that the symbols of interest for the remainder take the form
$$A_\ell(D)\left(\left(P_{j,k,h}(hH_1,hH_2,hH_3)(H_2+iH_3)^{k-j}\right)(H_2+iH_3)\right),$$
with $A_\ell(D) = (\partial_{p_1}\cdot\partial_{q_2})^\ell-\ell(\partial_{q_1}\cdot\partial_{p_2})(\partial_{p_1}\cdot\partial_{q_2})^{\ell-1}$.
By induction, we get the expected expression for the terms of order $h^{l}$ in the asymptotic expansion.
\end{proof}

As a corollary of this lemma and of the composition rule for pseudodifferential operators, we also find:
\begin{coro}\label{c:A:power-k} Let $k\geq 1$. Then
$$\Op_h^w((H_2+iH_3)^k)=A_h^k+\sum_{j=1}^{k}\Op_h^w\left(\tilde{P}_{j,k,h}(hH_1,hH_2,hH_3)\right)A_h^{k-j},$$
where $\tilde{P}_{j,k,h}(u,v,w)$ is a polynomial with coefficients depending polynomially on $h$ and smoothly on $(x,y,z)\in\IR^2\times\IS^1$ (with uniformly bounded derivatives).
\end{coro}

We now turn to the commutation properties of $A_h^k$ with the operators of interest for our analysis:
\begin{lemm}\label{l:commute-A-power-k} Let $Q$ and $W$ be two smooth functions on $\IR^2\times\IS^1$ whose derivatives are uniformly bounded.  Then, for every $k\geq 1$, 
$$[A_h^k,\Op_h^w(hQH_1)]=h\sum_{j=0}^{k-1}\Op_{h}^w(\mathbf{P}_{j,k,h}(hH_1,hH_2,hH_3))A_h^{k-1-j},$$
and
$$[A_h^k,\Op_h^w(W)]=h\sum_{j=0}^{k-1}\Op_{h}^w(\tilde{\mathbf{P}}_{j,k,h}(hH_1,hH_2,hH_3))A_h^{k-1-j},$$
where $\mathbf{P}_{j,k,h}(u,v,w)$ and $\tilde{\mathbf{P}}_{j,k,h}(u,v,w)$ are polynomials whose coefficients depend polynomially on $h$ and smoothly on $(x,y,z)\in \IR^2\times\IS^1$ (with derivatives that are uniformly bounded). 
\end{lemm}

\begin{proof} First, we observe that, thanks to Lemma~\ref{l:A:power-k}, one has
$$A_h^k=\Op_h^w\left((H_2+iH_3)^k+\sum_{j=1}^kP_{j,k,h}(hH_1,hH_2,hH_3)(H_2+iH_3)^{k-j}\right),$$
where $P_{j,k,h}$ are polynomials verifying the properties of the present Lemma. The second bracket formula is then a direct consequence of the composition rule for pseudodifferential operators (see Theorem~\ref{t:composition}) together with Corollary~\ref{c:A:power-k}. In fact, since $W$ depends only on $(x,y,z)$, the terms in the asymptotic expansion will only involves derivatives of the symbol of $A_h^k$ with respect to the variables $(\xi,\eta,\zeta)$.

We now turn to the first bracket which can be rewritten as
$$[A_h^k,\Op_h^w(hQH_1)]=\sum_{j=0}^k\left[\Op_h^w\left(P_{j,k,h}(hH_1,hH_2,hH_3)(H_2+iH_3)^{k-j}\right),\Op_h^w(QhH_1)\right],$$
with $P_{0,k,h}=1$.
Given $0\leq j\leq k$, one can apply the composition rule from Theorem~\ref{t:composition} to each term, i.e.
\begin{multline*}\left[\Op_h^w\left(P_{j,k,h}(hH_1,hH_2,hH_3)(H_2+iH_3)^{k-j}\right),\Op_h^w(QhH_1)\right]\\
=2\sum_{0\leq 2\ell\leq k}\frac{h^{2\ell+1}}{(2\ell+1)!}\Op_h^w\left(A(D)^{2\ell+1}\left((P_{j,k,h}(hH_1,hH_2,hH_3)(H_2+iH_3)^{k-j})(QhH_1)\right)\right).
\end{multline*}
Here, the sum over $\ell$ is bounded as we are only considering polynomials symbols in the variables $(\xi,\eta,\zeta)$ (with the total degree being bounded by $k+1$). Recalling the exact expression of $A(D)$ from Theorem~\ref{t:composition} and Corollary~\ref{c:A:power-k} (together with several applications of the composition formula), we find the expected result.
\end{proof}

\begin{rema}
Similar statements as those of Lemmas \ref{l:A:power-k}, \ref{l:commute-A-power-k} and Corollary \ref{c:A:power-k} hold for $A_h^*$ replacing $A_h$.
\end{rema}

\subsection{Inductive argument: proof of Proposition \ref{p:support_properties}}

In order to prove the localization properties \textit{(S.1)} and \textit{(S.2)} for the measure $\mu_\infty$, we start from the following semiclassical estimates. For every $k\geq 0$,
\begin{align}
\label{e:microlocal_estimate_1}
& \left\langle \Op_h^w(b(x,y,z,hH_1)\tilde{\chi}_R^B\tilde{\chi}_\varepsilon^C)^2A_h^k(\widehat{P}_{h,\lambda}-\lambda_h)u_h,A_h^ku_h\right\rangle = \ml{O}(h^\infty), \\[0.2cm]
\label{e:microlocal_estimate_2}
 &  \left\langle \Op_h^w(b(x,y,z,hH_1)\tilde{\chi}_R^B\tilde{\chi}_\varepsilon^C)^2(A^*_h)^k(\widehat{P}_{h,\lambda}-\lambda_h)u_h,(A^*_h)^ku_h\right\rangle = \ml{O}(h^\infty).
\end{align}
Using estimates \eqref{e:microlocal_estimate_1} and \eqref{e:microlocal_estimate_2} together with~\eqref{e:AAstar} and the symbolic calculus developed in Section \ref{s:preliminary_section}, we aim at deriving inductively suitable concentration properties for the distributions $\varrho^\pm_{k,\varepsilon,R,h}$ defined by
\begin{align*}
\varrho^+_{k,\varepsilon,R,h} & : b \mapsto \left\langle \Op_h^w(b(x,y,z,hH_1)\tilde{\chi}_R^B\tilde{\chi}_\varepsilon^C)A_h^ku_h,A_h^ku_h\right\rangle, \\[0.2cm]
\varrho_{k,\varepsilon,R,h}^-& : b \mapsto \left\langle \Op_h^w(b(x,y,z,hH_1)\tilde{\chi}_R^B\tilde{\chi}_\varepsilon^C)(A^*_h)^ku_h,(A^*_h)^k u_h\right\rangle.
\end{align*}
\begin{rema} Note that, in order to make sense of the limit measures for $k\geq 1$, one needs to have an a priori upper bound on 
$$\left\| \Op_h^w(b(x,y,z,hH_1)\tilde{\chi}_R^B\tilde{\chi}_\varepsilon^C)A_h^ku_h\right\|$$
which will be part of the argument below. For $k=1$, such an upper bound follows for instance from the a priori estimate~\eqref{e:apriori-semiclassical1} but, for $k\geq 2$, this does not longer work immediately.
\end{rema}

We will show along the process that, up to additional extractions, the weak limits of these distributions are well defined as non-negative Radon measures:
 $$
 \langle \varrho_{k}^\pm , b \rangle := \lim_{\varepsilon\to 0^+}\lim_{R \to + \infty} \lim_{h \to 0} \langle \varrho^\pm_{k,\varepsilon,h} , b \rangle,
 $$
and we will deduce from these measures the desired support properties of $\mu_\infty$. For the sake of exposition, we start with the first step $k = 0$ which is slightly easier to handle:

\begin{lemm}
\label{l:0_step_support} The measure $\mu_\infty= \varrho_{0}^{\pm}$ satisfies:
$$\operatorname{supp}(\mu_\infty)\subset\left\{ (q,E) \in \mathcal{M} \times \IR \, : \, -\frac{\lambda_0-W}{1-Q}\leq  E  \leq \frac{\lambda_0-W}{1+Q}\right\}.$$
In particular, the support of the measure $\mu_\infty$ is compact in the $E$ variable  and disjoint with the classical forbidden region $\{W>\lambda_0\}$. Moreover, 
\begin{equation}
\label{e:support_1}
 \operatorname{supp}(\mu_\infty) \subset \mathcal{H}_{+}^{-1}(1) \cup \mathcal{H}_{-}^{-1}(1) \cup \big( \operatorname{supp} \varrho_{1}^+ \cap \operatorname{supp} \varrho_{1}^- \big),
\end{equation}
and 
$$\operatorname{supp} \varrho_{1}^\pm\subset\operatorname{supp}(\mu_\infty),$$ 
where we recall that $\mathcal{H}_{\pm}(q,E) = \pm E^{-1}(\lambda_0 - W - E Q)$.
\end{lemm}

\begin{rema}
Recall also from Appendix~\ref{a:spectral} that condition $\|Q\|_{\ml{C}^0}<1$ was initially imposed to ensure the hypoellipticity (and the semiboundedness) of the operator $\widehat{P}_h$.
\end{rema}

\begin{proof} Given $b\in\ml{C}^{\infty}_c(\ml{U}_0\times\IR)$, one has
\begin{equation}
\label{e:ground_estimate}
\left\langle \Op_h^w(b(x,y,z,hH_1)\tilde{\chi}_R^B\tilde{\chi}_\varepsilon^C)(\widehat{P}_{h,\lambda}-\lambda_h)u_h,u_h\right\rangle=\ml{O}(h^\infty).
\end{equation}
Recalling that
\begin{equation}
\label{e:first_A_decomposition}
\Op_h^w(H^2_2+H_3^2)=A_h^*A_h+h\Op_h^w(H_1) + h^2c_0,
\end{equation}
one can use the composition rule for pseudodifferential operators together with the a priori estimate~\eqref{e:apriori-semiclassical1}. This yields
\begin{align*}
\big \langle \Op_h^w(b(x,y,z,hH_1)(hH_1+hQH_1+W-\lambda_h)\tilde{\chi}_R^B\tilde{\chi}_\varepsilon^C)u_h,u_h\big \rangle & = - \langle \varrho^+_{1,\varepsilon,R,h} , b \rangle + \ml{O}(h). 
 \end{align*}
Thanks to~\eqref{e:apriori-semiclassical1} and to the Calder\'on-Vaillancourt Theorem, the right-hand side defines a bounded sequence in $\ml{D}'(\ml{U}_0\times\IR)$. Moreover, the Garding inequality~\eqref{e:Garding0} ensures that the limit distribution is a nonnegative Radon measure. Hence, letting $h\rightarrow 0^+$ and $R\rightarrow+\infty$ (in this order), one finds that, for every $b$ compactly supported in $\ml{U}_0\times\IR$,
$$\mu^\varepsilon\left(b(x,y,z,E)(E(1+Q)+W-\lambda_0)\right)= - \varrho^+_{1,\varepsilon}(b).$$
From this, one infers that, on the support of $\mu^\varepsilon$, $(1+Q)E + W - \lambda_0 \leq 0$, and moreover that 
\begin{equation}
\label{e:support_intersection_1}
 \operatorname{supp}(\mu^\varepsilon) \setminus \mathcal{H}_+^{-1}(1) = \operatorname{supp} \varrho^+_{1,\varepsilon} .
\end{equation}
Similarly, using now the identity
\begin{equation}
\label{e:second_A_decomposition}
\Op_h^w(H^2_2+H_3^2)=A_h A_h^* - h\Op_h^w(H_1) + h^2 c_0
\end{equation}
instead of \eqref{e:first_A_decomposition}, and using again \eqref{e:ground_estimate}, one finds
\begin{align*}
\big \langle \Op_h^w(b(x,y,z,hH_1)(-hH_1+hQH_1+W-\lambda_h)\tilde{\chi}_R^B\tilde{\chi}_\varepsilon^C)u_h,u_h\big \rangle & = - \langle \varrho_{1,\varepsilon,h}^- , b \rangle + \ml{O}(h),
 \end{align*}
and thus
$$
\mu^\varepsilon\left(b(x,y,z,E)(-E(1-Q)+W-\lambda_0)\right)= - \varrho^-_{1,\varepsilon}(b).$$
This implies that $-E(1-Q) + W - \lambda_0 \leq 0$ on the support of $\mu^\varepsilon$ and moreover that 
\begin{equation}
\label{e:support_intersection_2}
  \operatorname{supp}(\mu^\varepsilon) \setminus \mathcal{H}_-^{-1}(1) = \operatorname{supp} \varrho^-_{1,\varepsilon} .
\end{equation}
Putting together \eqref{e:support_intersection_1} and \eqref{e:support_intersection_2} and, letting $\varepsilon\to 0^+$ this concludes the proof.
\end{proof}

We now turn to the general case for which we cannot make use of the a priori estimate~\eqref{e:apriori-semiclassical1} directly:

\begin{lemm}
\label{l:k_step_support} For every $ k \geq 0$ and for every $R>1$ and $\varepsilon>0$, the family $(\varrho_{k,R,\varepsilon,h}^\pm)_{0<h\leq h_0}$ is bounded in $\ml{D}'(\ml{U}_0\times\IR)$ and any accumulation point (as $h\to 0^+$) $\varrho_{k,R,\varepsilon}^\pm$ is a finite nonnegative Radon measure. Moreover, $(\varrho_{k,R,\varepsilon}^\pm)_{R,\varepsilon}$ is bounded and any accumulation point as $R\rightarrow+\infty$ and $\varepsilon\rightarrow 0^+$ (in this order) is a finite nonnegative Radon measure verifying
$$
\operatorname{supp} \varrho_{k}^\pm \subset \{ (q,E) \in \mathcal{M} \times \IR \, : \, \pm E (2k +1 \pm Q) + W - \lambda_0 \leq 0 \},
$$
Moreover, for every $k\geq 0$, $\operatorname{supp} \varrho_{k+1}^{\pm}\subset\operatorname{supp} \varrho_{k}^{\pm}$ and
$$
 \operatorname{supp} \varrho_{k}^{\pm}\setminus \operatorname{supp} \varrho_{k+1}^\pm \subset   \mathcal{H}_\pm^{-1}(2k+1) .
$$
\end{lemm}

\begin{proof}
The case $k = 0$ follows by Lemma \ref{l:0_step_support}. Assume that the claim holds for every $0\leq j\leq k$ and moreover that the following a priori estimates hold, for any $b\in\ml{C}^{\infty}_c(\ml{U}_0\times\IR)$ and for any $h>0$ small enough,
\begin{equation}
\label{e:norm_estimate_Ak}
\left\| \Op_h^w(b(x,y,z,hH_1)\tilde{\chi}_{2^{j-(k+1)}R}^B\tilde{\chi}_{2^{k+1-j}\varepsilon}^C)A^j_h u_h \right\|_{L^2} \leq C\left(\|b\|_\infty+\ml{O}_{R,\varepsilon,b}(h^{\frac{1}{2}})\right),
\end{equation}
for $0 \leq j \leq k$ and for some constant $C>0$ that is independent of $\varepsilon$ and $R$ and that depends on the support of $b$ (and also on $W$ and $Q$). Moreover, the constant in the remainder depends on a finite number of derivatives of $b$. Let us prove the claim for $k+1$. 
\begin{rema} We change the parameters in our cutoff functions at each stage but it does not hurt the argument as bounds like~\eqref{e:norm_estimate_Ak} holds for any choice of $R,\varepsilon$ and do not depend on the choice of subsequence. In the end, recall that we are aiming in the end at taking the limit $h\rightarrow 0^+$, $R\rightarrow +\infty$ and $\varepsilon\rightarrow 0^+$ (in this order). For instance, all along the analysis, one could take the sequences $(R_n)_{n\geq 1}$ and $(\varepsilon_m)_{m\geq 1}$ to be of the form $R_n=2^n$ and $\varepsilon_m=2^{-m}$ in view of matching the inductive process chosen here. We keep the notations $R\rightarrow+\infty$ and $\varepsilon\rightarrow 0^+$ to alleviate notations. 
\end{rema}

We begin by proving the same a priori estimate on
$$\Vert \Op_h^w(b(x,y,z,hH_1)\tilde{\chi}_R^B\tilde{\chi}_\varepsilon^C)A^{k+1}_h u_h \Vert_{L^2},$$
which is the main technical point of the analysis. To do that, we begin with equality~\eqref{e:microlocal_estimate_1} which can be expanded as follows
\begin{align*}
 \left\langle \Op_h^w(b(x,y,z,hH_1)\tilde{\chi}_R^B\tilde{\chi}_\varepsilon^C)^2A_h^k\Op_h^w(hH_1+hQH_1+W+h^2(W_{1,\lambda}+c_0)-\lambda_h)u_h,A_h^ku_h\right\rangle & \\[0.2cm] & \hspace*{-11cm} =-\left\langle \Op_h^w(b(x,y,z,hH_1)\tilde{\chi}_R^B\tilde{\chi}_\varepsilon^C)^2A_h^kA_h^*A_h u_h,A_h^ku_h\right\rangle 
 + \ml{O}(h^\infty).
\end{align*}
Applying Lemma~\ref{l:commute-A-power-k} together with~\eqref{e:commutator_A_A*} $k$ times, we can deduce from the support properties of our symbol and from~\eqref{e:norm_estimate_Ak} that
\begin{multline*}
 \left\langle \Op_h^w(b(x,y,z,hH_1)\tilde{\chi}_R^B\tilde{\chi}_\varepsilon^C)^2\Op_h^w((2k+1)hH_1+hQH_1+W-\lambda_h) A_h^ku_h,A_h^ku_h\right\rangle \\[0.2cm] =-\left\langle \Op_h^w(b(x,y,z,hH_1)\tilde{\chi}_R^B\tilde{\chi}_\varepsilon^C)^2A_h^*A_h^{k+1} u_h,A_h^ku_h\right\rangle 
 + \ml{O}_{R,\varepsilon}(h).
\end{multline*}
Here, the fact that we can control the remainder terms at each step by~\eqref{e:norm_estimate_Ak} follows from elliptic estimates for pseudodifferential operators as in~\cite[Th.~E.33]{DyatlovZworski19} and from the fact that we schrink the support of our cutoffs functions at each stage of the induction.

Applying the composition rule for pseudodifferential operators together with the induction hypothesis, we find that
\begin{multline*}
 \left\langle \Op_h^w((b(x,y,z,hH_1)\tilde{\chi}_R^B\tilde{\chi}_\varepsilon^C)^2((2k+1)hH_1+hQH_1+W-\lambda_h) A_h^ku_h,A_h^ku_h\right\rangle \\[0.2cm] =-\left\|\Op_h^w(b(x,y,z,hH_1)\tilde{\chi}_R^B\tilde{\chi}_\varepsilon^C)A_h^{k+1} u_h\right\|^2 
 + \ml{O}_{R,\varepsilon}(h).
\end{multline*}
By induction, the left-hand side is bounded from which we deduce the expected upper bound at step $k+1$.

We note that the upper bound~\eqref{e:norm_estimate_Ak} allows to verify by induction that any accumulation point $\varrho_{k,\varepsilon,R}^\pm$ (as $h\rightarrow 0^+$) is a finite nonnegative Radon measure whose total mass is independent of $\varepsilon$ and $R$. We can thus take the limit $R\rightarrow +\infty$ and $\varepsilon\rightarrow 0^+$ in this order. We obtain the expected limit measure $\varrho_{k+1}^\pm$ and we can now derive its support properties. Repeating the same argument with $\Op_h^w(b(x,y,z,hH_1)\tilde{\chi}_R^B\tilde{\chi}_\varepsilon^C)$ instead of $\Op_h^w(b(x,y,z,hH_1)\tilde{\chi}_R^B\tilde{\chi}_\varepsilon^C)^2$ (and with $(A_h^*)^k$ instead of $A_h^k$) and taking the limits $h\rightarrow 0^+$ and $R\rightarrow +\infty$ (in this order), we can prove that
\begin{equation}
\label{e:link_semiclassical_measures}
\varrho_{k}^{\pm}\left(b(x,y,z,E)(\pm E(2k+1 \pm Q)+W-\lambda_0)\right)= - \varrho^\pm_{k+1}(b).
\end{equation}
Finally, from \eqref{e:link_semiclassical_measures}, we get the first statement of the lemma and moreover:
\begin{align*}
\operatorname{supp} \varrho_{k}^{\pm} \setminus \mathcal{H}_\pm^{-1}(2k+1) = \operatorname{supp} \varrho_{k+1}^{\pm}.
\end{align*} 
This concludes the proof.
\end{proof}

Finally, Lemma \ref{l:k_step_support} implies that:
$$
\operatorname{supp}\mu_\infty\subset ( \mathcal{M}_{\lambda_0,W} \times \{ 0 \}) \cup  \bigcup_{k=0}^\infty \left(  \mathcal{H}_+^{-1}(2k+1) \cup \mathcal{H}_-^{-1}(2k+1) \right).
$$
Defining, for $k \in \mathbb{Z}_+$, 
$$
\mu_{k,\infty}^\pm := \mathbf{1}_{\mathcal{H}_\pm^{-1}(2k+1)} \mu_\infty,\quad \text{and}\quad \overline{\mu}_\infty := \mathbf{1}_{\mathcal{M}_{\lambda_0,W} \times \{0 \}} \mu_\infty,
$$ 
we obtain the proof of Proposition \ref{p:support_properties}.

\section{Normal form reduction}\label{s:normalform}

In order to study the invariance properties of the measure $\mu_\infty$, we will require a normal form procedure in the subelliptic regime $1 \ll |H_1| \lesssim h^{-1}$. This will allow us to work with functions that are adapted to the geometry of the problem. Roughly speaking, it amounts to work in a system of asymptotically symplectic coordinates as $|H_1|\rightarrow \infty$. This normal form approach was pioneered in Melrose's works~\cite{Melrose85} and recently revisited in the context of sub-Riemannian Laplacians associated with $3$-dimensional contact flows~\cite{ColindeVerdiereHillairetTrelat15, ColindeVerdiereHillairetTrelat21} as we are dealing here. See also~\cite{ RaymondVuNgoc15} for earlier related normal forms procedure in the context of $2$-dimensional magnetic semiclassical Schr\"odinger operators and~\cite{HelfferKordyukovRaymondVuNgoc16, Morin19, Morin20} regarding higher dimensional normal forms for magnetic operators.

Compared with~\cite{ColindeVerdiereHillairetTrelat15}, we will simplify some aspects of the normal form procedure. Rather than making a symplectic change of variables at each step of the iterative scheme and using the machinery of Fourier integral operators, we will just encode a slightly simpler change of variables (close to but not necessarily symplectic) into a small deformation of the test function 
$$
\mathbf{b} = b + \ml{O}\left( \frac{ \vert H_2 \vert + \vert H_3 \vert}{\vert H_1 \vert} \right)
$$ 
for the Wigner distribution $\mu_h^{R,\varepsilon}$ defined by \eqref{e:wigner-H1}. This deformation  makes the remainder term in $\{H_2^2+H_3^2, \mathbf{b}\}$ as small as possible in the regime $H_2^2+H_3^2\ll H_1^2$. We will need as well to introduce an auxiliary function $\mathbf{H}_1$ that behaves asymptotically as $H_1$ and that ensures $\{H_2^2+H_3^2,\mathbf{H}_1\}$ to be small in that same regime. This simplified version of the normal form method is in fact sufficient to obtain the desired invariance properties of the semiclassical measure $\mu_\infty = \lim_{\varepsilon \to 0} \mu^\varepsilon$. We emphasize that Wigner type distributions enjoy somehow more flexibility regarding change of variables than Fourier integral operators (since many negligible terms disappear in the limit $h \to 0$), and we will crucially exploit this fact to avoid the use of cumbersome symplectic changes of coordinates that in the end match with ours in the semiclassical limit (see \cite[Sect. 3.1]{ArnaizSun22} for a related construction involving  two-microlocal semiclassical measures).

\begin{rema} In this section, the fact that we have simple bracket formulas as~\eqref{e:commutators-hamiltonian} will make the proof slightly simpler and more explicit regarding the terms appearing in the normal form. Yet, the strategy would remain the same if $\{H_1,H_2\}$ and $\{H_1,H_3\}$ were more general linear combinations of $H_1$, $H_2$ and $H_3$ (as in the general contact case). In any case, the fact that these brackets do not exactly vanish is exactly where the situation is more involved than in the flat Heisenberg case~\cite{FermanianFischer21, FermanianLetrouit20} like when considering varying magnetic fields rather than constant ones. 
\end{rema}

\subsection{Solving polynomial cohomological equations} 

We first recall the algebraic relations given by~\eqref{e:commutators-hamiltonian} and producing the subelliptic structure of our problem:
\begin{equation}\label{e:keyrelation}
 \{H_1,H_{2}\}=-KH_3,\quad \{H_1, H_3\}=H_{2},\quad\{H_2,H_3\}=-H_1.
\end{equation}
In some sense the pair $(H_2,H_3)$ behaves like a system of coordinates in $T^*\IR$ for the classical harmonic oscillator, at least in the region $H_2^2+H_3^2\ll H_1^2$. With that observation in mind, it is convenient to introduce complex notations:
$$
Z:=H_{2}+iH_3\quad\text{and}\quad \overline{Z}:=H_{2}-iH_3,
$$
so that
$$H_{2}^2+H_3^2=|Z|^2.$$
In particular, a direct calculation shows that, for every $k\neq l$ in $\IZ_+$, one has
\begin{equation}\label{e:solution-cohomological-function}
\left\{|Z|^2, \frac{Z^k \overline{Z}^l}{2i(l-k)}\right\}=H_1 Z^k \overline{Z}^l,
\end{equation}
and the relations~\eqref{e:keyrelation} now become
\begin{equation}\label{e:complex-keyrelation}
 \{H_1,Z\}=iK_+Z+iK_-\overline{Z},\quad \{H_1, \overline{Z}\}=-iK_-Z-iK_+\overline{Z},\quad\{Z,\overline{Z}\}=2iH_1,
\end{equation}
where 
$$K_+:=\frac{1+K}{2}\quad\text{and}\quad K_-=\frac{1-K}{2}.$$
In particular, one has
\begin{equation}\label{e:H1commuteZ2}
 \{|Z|^2,H_1\}=iK_-(Z^2-\overline{Z}^2).
\end{equation}

\begin{rema}
 In the rest of this Section, we will always suppose that we work in the region $H_1\neq 0$. In the end, the formulas will only be used in the regime $|H_2|,|H_3|\ll |H_1|$.
\end{rema}

\subsection{A small deformation of $H_1$}

We start with the deformation of the variable $H_1$.

\subsubsection{Normal form procedure}

In view of~\eqref{e:H1commuteZ2} and~\eqref{e:solution-cohomological-function}, we set
$$P_2:=\frac{K_-}{2}\left(\left(\frac{H_2}{H_1}\right)^2-\left(\frac{H_3}{H_1}\right)^2\right).$$
By construction, one has
$$\{|Z|^2,H_1(1 + P_2)\}=\frac{(Z^2+\overline{Z}^2)}{4}\left\{|Z|^2,\frac{K_-}{H_1}\right\}.$$
Iterating this procedure one more time, one can find $P_3$ in $\ml{P}(\IR^2\times\IS^1)$ of the form
$$P_3=\sum_{|\alpha|=3}P_{3,\alpha}(x,y,z)\left(\frac{H_2}{H_1}\right)^{\alpha_2}\left(\frac{H_3}{H_1}\right)^{\alpha_3}$$
and such that, if we set
\begin{equation}\label{e:normalform-H1}
 \mathbf{H}_1:=H_1\left(1+P_2 +P_3\right),
\end{equation}
then one has
\begin{equation}\label{e:H1commuteZ2-normalform}
 \{H_2^2+H_3^2,\mathbf{H}_1\}=H_2^2 R_1+H_3^3 R_2+H_2H_3 R_3,
\end{equation}
with $R_j$ belonging in $\mathcal{P}(\IR^2\times\IS^1)$ and being of the form
$$R_j=\sum_{|\alpha|\geq 2}R_{j,\alpha}(x,y,z)\left(\frac{H_2}{H_1}\right)^{\alpha_2}\left(\frac{H_3}{H_1}\right)^{\alpha_3}.$$
Hence, in the regime $H_2^2+H_3^2\ll H_1^2$, this Poisson bracket is somehow of smaller order than the one appearing in~\eqref{e:H1commuteZ2}.
\medskip

We finally state the following analogue of Lemma~\ref{l:symbol-cone}:
\begin{lemm}\label{l:symbol-cone-normalform} Let $\mathcal{K
}$ be a compact subset of $\IR^2\times\IS^1$. For every $0<\varepsilon\leq 1$, for every $N_0\geq 2$ and for every $(\alpha,\beta)\in\IZ_+^6$, one can find a constant $C_{\varepsilon, N_0, \mathcal{K}, \alpha,\beta}$ such that, for every $(q,p)=(x,y,z,\xi,\eta,\zeta)$ in $C_{2^{N_0}\varepsilon}(\mathcal{K})\setminus C_{2^{-N_0}\varepsilon}(\mathcal{K})$, one has, for every $j\in\{2,3\}$, 
$$\left|\partial^\alpha_{q}\partial_{p}^\beta \left(\frac{\mathbf{H}_1}{\sqrt{1+H_2^2+H_3^2}}\right)\right|\leq C_{\varepsilon, N_0, \mathcal{K},\alpha,\beta}\langle p\rangle^{-|\beta|},$$ 
where $\langle p\rangle:=(1+\xi^2+\eta^2+\zeta^2)^{\frac{1}{2}}.$ 
\end{lemm}
\begin{proof}
This is a direct consequence of Corollary~\ref{c:derivative-cone} and Lemma~\ref{l:symbol-cone}.
\end{proof}

\subsubsection{Rewriting $\mu_{h}^{R,\varepsilon}$ using $\mathbf{H}_1$}

We now come back to the distribution $\mu_{h}^{R,\varepsilon}$ that was introduced in~\eqref{e:wigner-H1} and defined using $H_1$. We next show that we can replace $H_1$ by $\mathbf{H}_1$ modulo small remainders in $h$ and $\varepsilon$.
\begin{lemm}\label{l:wigner} For $b$ in $\ml{C}^\infty_c(\ml{U}_0\times\IR)$, we set
\begin{equation}
\label{e:black_mu}
\langle \bm{\mu}_{h}^{R,\varepsilon},b\rangle:=\left\langle \Op_h(b(x,y,z,h\mathbf{H}_1)\tilde{\chi}_\varepsilon^C\tilde{\chi}_\varepsilon^{\mathbf{C}}\tilde{\chi}_R^B\tilde{\chi}_R^{\mathbf{B}})u_h,u_h\right\rangle_{L^2},
\end{equation}
where
 $$\tilde{\chi}_R^{\mathbf{B}}:=\tilde{\chi}\left(\frac{\mathbf{H}_1^2+H_2^2+H_3^2}{4R}\right),\quad\text{and}\quad\tilde{\chi}_\varepsilon^{\mathbf{C}}:=\tilde{\chi}\left(\frac{\varepsilon \mathbf{H}_1}{4\sqrt{H_2^2+H_3^2+1}}\right).$$
 Then, one has, as $h\rightarrow 0^+$, $R\rightarrow+\infty$ and $\varepsilon\rightarrow 0$ (in this order),
 $$\langle\bm{\mu}_{h}^{R,\varepsilon},b\rangle=\langle\mu_{h}^{R,\varepsilon},b\rangle+\ml{O}_{b,\varepsilon,R}(h)+\ml{O}_{\varepsilon}(R^{-1})+\ml{O}_b(\varepsilon^2).$$
\end{lemm}
\begin{rema}\label{r:support}
 Note that, on the support of $\tilde{\chi}_\varepsilon^C$, one has $\mathbf{H}_1=H_1(1+\ml{O}(\varepsilon^2))$ so that $\tilde{\chi}_\varepsilon^C=1$ on the support of $\tilde{\chi}_\varepsilon^C\tilde{\chi}_\varepsilon^{\mathbf{C}}$ if $\varepsilon>0$ is chosen small enough. We just keep the function $\tilde{\chi}_\varepsilon^C$ to ensure that $\mathbf{H}_1$ is well defined. The same holds for $\tilde{\chi}_R^B\tilde{\chi}_R^{\mathbf{B}}$.
\end{rema}

\begin{proof} Let $b(x,y,z,E)$ be an element in $\ml{C}^\infty_c(\ml{U}_0\times\IR^*)$. One has
$$\langle\mu_{h}^{R,\varepsilon},b\rangle=\left\langle \Op_h(b(x,y,z,hH_1)\tilde{\chi}_\varepsilon^C\tilde{\chi}_R^B)u_h,u_h\right\rangle_{L^2},$$
where $\tilde{\chi}_\varepsilon^C$ was defined in~\eqref{e:cutoff-cone} as
$$\tilde{\chi}_\varepsilon^C:=\tilde{\chi}\left(\frac{\varepsilon H_1}{\sqrt{H_2^2+H_3^2+1}}\right).$$
Arguing as in the proof of Lemma~\ref{l:symbol}, one knows that
$$\tilde{\chi}_\varepsilon^C\tilde{\chi}\left(\frac{\varepsilon \mathbf{H}_1}{4\sqrt{H_2^2+H_3^2+1}}\right)b(x,y,z,hH_1)$$
belongs to $S^0_{\text{cl}}(T^*(\IR^2\times\IS^1))$ with all seminorms uniformly bounded in terms of $0<h\leq 1$ (but not on $\varepsilon$ a priori).

Set now $\chi_\varepsilon^{\mathbf{C}}=1-\tilde{\chi}_\varepsilon^{\mathbf{C}}$. On the support of $b(x,y,z,hH_1)\tilde{\chi}_\varepsilon^C\chi_\varepsilon^{\mathbf{C}}$, one has 
$$1\leq\frac{\varepsilon |H_1|}{\sqrt{1+H_2^2+H_3^2}}\leq 10,$$
for small enough $\varepsilon>0$. In particular, we can argue as in~\eqref{e:measure-pseudo} and write
$$b(x,y,z,hH_1)\tilde{\chi}_R^B\tilde{\chi}_\varepsilon^C\chi_\varepsilon^{\mathbf{C}}=\frac{b(x,y,z,hH_1)\tilde{\chi}_R^B\tilde{\chi}_\varepsilon^C\chi_\varepsilon^{\mathbf{C}}}{1+H_2^2+H_3^2}(1+H_2^2+H_3^2),$$
where the first term of the product on the right-hand side belongs to $S^{-2}_{\text{cl}}$ thanks to the above support properties. Using the Calder\'on-Vaillancourt Theorem together with~\eqref{e:eigenvalue-conjugation} and~\eqref{e:apriori-semiclassical1}, we find that
$$\langle\mu_{h}^{R,\varepsilon},b\rangle=\left\langle \Op_h(b(x,y,z,hH_1)\tilde{\chi}_\varepsilon^C\tilde{\chi}_\varepsilon^{\mathbf{C}}\tilde{\chi}_R^B)u_h,u_h\right\rangle_{L^2}+\ml{O}_{R,\varepsilon}(h)+\ml{O}_{\varepsilon}(R^{-1}).$$
Similarly, we can insert the cutoff function $\tilde{\chi}_R^{\mathbf{B}}$ if we note that, on the support of $\chi_R^{\mathbf{B}}\tilde{\chi}_R^B$, one has
$$\frac{1}{10}R\leq H_1^2+H_2^2+H_3^2\leq 10R,$$
hence the involved symbol is compactly supported. We can then apply the exact same argument and show that
$$\langle\mu_{h}^{R,\varepsilon},b\rangle=\left\langle \Op_h(b(x,y,z,hH_1)\tilde{\chi}_\varepsilon^C\tilde{\chi}_\varepsilon^{\mathbf{C}}\tilde{\chi}_R^{\mathbf{B}}\tilde{\chi}_R^B)u_h,u_h\right\rangle_{L^2}+\ml{O}_{R,\varepsilon}(h)+\ml{O}_{\varepsilon}(R^{-1}).$$

We are now left with replacing $H_1$ by $\mathbf{H}_1$ in the last component of $b$. We write 
$$b(x,y,z,hH_1)=b(x,y,z,h\mathbf{H}_1)+h(H_1-\mathbf{H}_1)\int_0^1\partial_Eb(x,y,z,h\mathbf{H}_1+th(H_1-\mathbf{H}_1))dt.$$
Hence, applying Corollary~\ref{c:derivative-cone} together with~\eqref{e:apriori-semiclassical1} and~\eqref{e:normalform-H1}, the composition rule for pseudodifferential operators and the Calder\'on-Vaillancourt Theorem, one finds the expected conclusion.
\end{proof}

\subsection{A small deformation of $a$} We now proceed similarly and introduce a small deformation of $a$ whose Poisson bracket with $H_2^2+H_3^2$ is small in the regime $H_2^2+H_3^2\ll H_1^2$. Let $a$ be a smooth function compactly supported on $\ml{U}_0$, we write
$$
\{|Z|^2,a\}=Z\{\overline{Z},a\}+\overline{Z}\{Z,a\}.
$$
In the region $H_1\neq 0$ and in view of~\eqref{e:solution-cohomological-function}, we can set
$$a_1:=\frac{Z}{2iH_1}\{\overline{Z},a\}-\frac{\overline{Z}}{2iH_1}\{Z,a\}.$$
We then find
$$
\{|Z|^2,a + a_1\}=\frac{Z}{2i}\left\{|Z|^2,\frac{\{\overline{Z},a\}}{H_1}\right\}-\frac{\overline{Z}}{2i}\left\{|Z|^2,\frac{\{Z,a\}}{H_1}\right\}.
$$
Recalling that $a$ is a function on $\ml{U}_0$, one has 
\begin{eqnarray*}
\{Z,\{\overline{Z},a\}\}-\{Z,\{\overline{Z},a\}\}&=&(X_\perp+iV)(X_\perp-iV)(a)-(X_\perp-iV)(X_\perp+iV)(a)\\[0.2cm]
&=&2i[V,X_\perp](a)=2iX(a).
 \end{eqnarray*}
Hence, letting $X_{Z}$ (resp. $X_{\overline{Z}}$) be the (complex) vector field generated by $Z$ (resp. $\overline{Z}$) the above expression can be simplified as
\begin{align*}
 \{|Z|^2,a+a_1\} & \\[0.2cm]
  & \hspace*{-1.5cm} =\frac{|Z|^2}{H_1}X(a)+\frac{(Z^2X_{\overline{Z}}^2-\overline{Z}^2X_Z^2)(a)}{2iH_1}-\frac{Z^2X_{\overline{Z}}(a)}{2iH_1^2}\{\overline{Z},H_1\}+\frac{\overline{Z}^2X_Z(a)}{2iH_1^2}\left\{Z,H_1\right\}.
\end{align*}
We would now like to eliminate the terms of magnitude $1/H_1$ that can be averaged via our cohomological equations. Namely, we set
$$a_2=-\frac{(Z^2X_{\overline{Z}}^2+\overline{Z}^2X_Z^2)(a)}{8H_1^2}.$$
This allows us, after defining $\mathbf{a}:=a+a_1 + a_2$, to get
\begin{equation}\label{e:poisson-bracket-a-normal-form}
 \{H_2^2+H_3^2,\mathbf{a}\}=\frac{H_2^2+H_3^2}{H_1}X(a)+\frac{1}{H_1}\mathbf{R}_a
 \end{equation}
 where $\mathbf{R}_a$ is an element of the form
 $$\mathbf{R}_a=\sum_{|\alpha|\geq 2}R_{a,\alpha}(x,y,z)\left(\frac{H_2}{H_1}\right)^{\alpha_2}\left(\frac{H_3}{H_1}\right)^{\alpha_3},$$
 with $R_{a,\alpha}$ that is compactly supported in $\mathcal{U}_0$.

For simplicity of exposition, we will use more compact notations for $\mathbf{a}$: 
\begin{equation}\label{e:convention-a}
\mathbf{a}=\sum_{|\alpha|\leq 2} a_\alpha\left(\frac{H_2}{H_1}\right)^{\alpha_2}\left(\frac{H_3}{H_1}\right)^{\alpha_3},
\end{equation}
where the functions $a_\alpha$ are also defined on $\ml{U}_0$ and explicitely given by
\begin{equation}\label{e:coeff-a-normal-form}
a_{(0,0)}:=a,\ a_{(1,0)}:=-V(a),\ a_{(0,1)}:=X_\perp(a),\
\end{equation}
and
\begin{equation}\label{e:coeff-a-normal-form2}
a_{(2,0)}:=-a_{(0,2)}:=-\frac{(X_\perp^2-V^2)(a)}{4},\quad  a_{(1,1)}:=-\frac{(X_\perp V+VX_\perp)(a)}{2}. 
\end{equation}

\section{Invariance properties}\label{s:invariance}

In view of Lemma~\ref{l:wigner}, we are left to study the Wigner type distribution $\bm{\mu}_{h}^{R,\varepsilon}$ given by \eqref{e:black_mu}. Our goal is to prove that any accumulation point $\mu_\infty$ of this sequence as $h\rightarrow 0$, $R\rightarrow+\infty$ and $\varepsilon\rightarrow 0$ (in this order) verifies certain invariance properties: 

\begin{prop}
\label{p:invariance} Let $(\psi_h,\lambda_h)$ be a sequence satisfying \eqref{e:semiclassical-eigenvalue} and set 
$$\mathcal{H}_1 (q,E):=\lambda_0-W(q)-EQ(q),
$$
and 
\begin{equation}\label{e:vectorfield-invariance}
\mathcal{X}_{W,Q}:=(\lambda_0-W)X+ \Omega_{\mathcal{H}_1}+EX(\ml{H}_1)\partial_E.
\end{equation}
 Let $\mu_\infty$ be any semiclassical measure obtained as a weak limit for the sequence of distributions $\bm{\mu}_h^{R,\varepsilon}$ defined from the sequence $(\psi_h,\lambda_h)$. Then, for every $b \in \mathcal{C}^1_c(\mathcal{U}_0 \times \IR)$,
\begin{align}
\label{e:invariance_desplayed}
\int_{\ml{U}_0\times\IR} \big( \mathcal{X}_{W,Q} (b) +X(\mathcal{H}_1) b \big) d\mu_\infty=0.
\end{align}
\end{prop}

In order to prove Proposition \ref{p:invariance}, we start by fixing an element $b$ in $\mathcal{C}^\infty_c(\ml{U}_0 \times \IR)$. Rather than looking at $b$ directly, we will consider test functions based on the normal form $\mathbf{b}$ from~\eqref{e:convention-a}, that is,
\begin{equation}
\label{e:normal_form_for_b}
\mathbf{b}(x,y,z,E) = \sum_{\vert \alpha \vert \leq 2} b_\alpha(x,y,z,E) \left(\frac{H_2}{H_1}\right)^{\alpha_2}\left(\frac{H_3}{H_1}\right)^{\alpha_3},
\end{equation}
where $b_\alpha$ is given by \eqref{e:coeff-a-normal-form} and \eqref{e:coeff-a-normal-form2} with $b$ instead of $a$ (the variable $E$ plays the role of a parameter).
One then has

$$\left\langle\left[\widehat{P}_{h,\lambda}, \Op_h^w\left(\mathbf{H}_1 \mathbf{b}(\cdot,h \mathbf{H}_1)\tilde{\chi}_\varepsilon^C\tilde{\chi}_\varepsilon^{\mathbf{C}}\tilde{\chi}_R^B\tilde{\chi}_R^{\mathbf{B}}\right)\right] u_h,u_h\right\rangle=\ml{O}(h^\infty).$$
On the other hand, we know that $\mathbf{H}_1 \mathbf{b}(\cdot,\cdot,\cdot,h \mathbf{H}_1)\tilde{\chi}_\varepsilon^C\tilde{\chi}_\varepsilon^{\mathbf{C}}\tilde{\chi}_R^B\tilde{\chi}_R^{\mathbf{B}}$ belongs to $S^1_{\text{cl}}$ while the symbol of $\widehat{P}_{h,\lambda}$ lies in $S^2_{\text{cl}}$. Then, by the composition rules for the Weyl quantization (see Appendix~\ref{a:pdo}), one finds
\begin{equation}
\label{e:Wignerequation}
\big \langle\Op_h^w\left(\left\{P_{h},\mathbf{H}_1 \mathbf{b}(\cdot,h \mathbf{H}_1)\tilde{\chi}_\varepsilon^C\tilde{\chi}_\varepsilon^{\mathbf{C}}\tilde{\chi}_R^B\tilde{\chi}_R^{\mathbf{B}}\right\}\right) u_h,u_h\big\rangle=\ml{O}(h^2),
\end{equation}
where $P_h = H_2^2+H_3^2+hQH_1+W$.
Recall that, as before, the measure in the $L^2$ scalar product is the standard Lebesgue measure thanks to our conventions for $\widehat{P}_{h,\lambda}$ and $u_h$ (see \eqref{e:conjugation-Ph}, \eqref{e:conjugated-eigenfunction} and \eqref{e:eigenvalue-conjugation}).

\subsection{Removing the derivatives of the cutoff function}

Before exploiting the normal form procedure, we start with the following Lemma in view of removing the cutoff functions $\tilde{\chi}_\varepsilon^C\tilde{\chi}_\varepsilon^{\mathbf{C}} \tilde{\chi}_R^B\tilde{\chi}_R^{\mathbf{B}}$ from the Poisson bracket:
\begin{lemm}\label{l:remove-cutoff} With the above conventions, one has:
\begin{equation}
\label{e:reduced_wigner_equation}
\left\langle \Op_h^w\big( \{P_h, \mathbf{H}_1\mathbf{b}(\cdot,h\mathbf{H}_1)\} \bm{\chi}_{R,\varepsilon} \big )u_h,u_h\right\rangle_{L^2(\operatorname{Leb})}=\ml{O}(\varepsilon^2)+\ml{O}_{R,\varepsilon}(h^{1/2}),
\end{equation}
where $\bm{\chi}_{R,\varepsilon} := \tilde{\chi}_\varepsilon^C\tilde{\chi}_\varepsilon^{\mathbf{C}} \tilde{\chi}_R^B\tilde{\chi}_R^{\mathbf{B}}$.
\end{lemm}
\begin{proof} In light of~\eqref{e:Wignerequation}, proving this equality amounts to show that
 $$\left\langle \Op_h^w\left(\mathbf{H}_1\mathbf{b}(\cdot,h\mathbf{H}_1)\{H_2^2+H_3^2,\bm{\chi}_{R,\varepsilon}\}\right)u_h,u_h\right\rangle_{L^2(\text{Leb})}=\ml{O}(\varepsilon^2)+\ml{O}_{R,\varepsilon}(h^{1/2}).$$
 The main observation is that, when considering $\{H_2^2+H_3^2,\bm{\chi}_{R,\varepsilon}\}$, the symbol becomes either compactly supported in $(\xi,\eta,\zeta)$ (if one differentiates $\tilde{\chi}_R^{\mathbf{B}}$), or supported in a region where 
 $$\frac{1}{10}\sqrt{1+H_2^2+H_3^2}\leq \varepsilon|H_1|\leq 10\sqrt{1+H_2^2+H_3^2}$$
 (if one differentiates $\tilde{\chi}_R^{\mathbf{C}}$). In the first case, one ends up with the bracket $\frac{1}{R}\{H_2^2+H_3^3,\mathbf{H}_1\}$ multiplied by $\mathbf{b}(\cdot,h\mathbf{H}_1)\mathbf{H}_1$. Modulo multiplication by $H_2$ or $H_3$, this is a bounded symbol whose supremum is of order $\ml{O}(\varepsilon^2)$ (with the constant being independent of $R$) thanks to~\eqref{e:H1commuteZ2-normalform} and to the compact support property. Applying the composition rule one more time and the Calder\'on-Vaillancourt Theorem together with the a priori estimate~\eqref{e:apriori-semiclassical1}, it yields the expected upper bound for these terms.
 
 It remains to deal with the case where one differentiates $\tilde{\chi}_R^{\mathbf{C}}$. In that case, we cannot exploit the fact that the support is bounded but we can remark that we end up with terms that are of the form
 $$\frac{\varepsilon \mathbf{b}(\cdot,h\mathbf{H}_1)\mathbf{H}_1}{\sqrt{1+H_2^2+H_3^2}}\{H_2^2+H_3^3,\mathbf{H}_1\}.
 $$
 Using Lemma~\ref{l:symbol-cone}, the first term lies in the class of symbol $S^0_{\text{cl}}$ inside the support of our symbol (recall that we have differentiated $\tilde{\chi}_{\varepsilon}^{\mathbf{C}}$) and is uniformly bounded in terms of $\varepsilon$ and $R$ (but not small a priori). Hence, we are left with showing that the term resulting from the Poisson bracket is indeed small. Again, we apply~\eqref{e:H1commuteZ2-normalform} together with the semiclassical a priori estimate~\eqref{e:apriori-semiclassical1} and we conclude thanks to the composition rule for pseudodifferential operators and the Calder\'on-Vaillancourt Theorem.
\end{proof}

\subsection{Proof of Proposition \ref{p:invariance}}

Let $b \in \mathcal{C}_c^\infty(\mathcal{U}_0 \times \IR)$, we consider as before the small deformation $\mathbf{b}(x,y,x,E)$ of $b$ given by \eqref{e:normal_form_for_b}. Thanks to Lemma~\ref{l:remove-cutoff}, one has
\begin{equation}\label{e:Wignerequation2}
\big \langle\Op_h^w\left(\left\{P_{h},\mathbf{H}_1 \mathbf{b}(\cdot,h \mathbf{H}_1)\right\}\tilde{\chi}_\varepsilon^C\tilde{\chi}_\varepsilon^{\mathbf{C}}\tilde{\chi}_R^B\tilde{\chi}_R^{\mathbf{B}}\right) u_h,u_h\big\rangle=\ml{O}(\varepsilon^2)+\ml{O}_{R,\varepsilon}(h^{\frac{1}{2}}).
\end{equation}
Hence, we need to understand
\begin{align}\label{e:normalform-semiclassical}
\{P_h,  \mathbf{H}_1 \mathbf{b} \}= \{H_2^2+H_3^2, \mathbf{H}_1 \mathbf{b}\}+\{ W ,  \mathbf{H}_1 \mathbf{b} \} + \{ h H_1 Q,  \mathbf{H}_1 \mathbf{b} \}.
\end{align}
Recalling~\eqref{e:H1commuteZ2-normalform} and~\eqref{e:poisson-bracket-a-normal-form}, one has
$$\{ H_2^2+H_3^2,\mathbf{H}_1\}= R_1H_2^2+R_2H_3^2+R_3H_2H_3,$$
and 
$$\{H_2^2+H_3^2,\mathbf{b}\}=\frac{H_2^2+H_3^2}{H_1}X(b)+\frac{1}{H_1}\mathbf{R}_b+h\{ H_2^2+H_3^2,\mathbf{H}_1\}\sum_{|\alpha|\leq 2}\partial_Eb_\alpha\left(\frac{H_2}{H_1}\right)^{\alpha_2}\left(\frac{H_3}{H_1}\right)^{\alpha_3},$$
where $\mathbf{R}_b$ and $(R_j)_{j\geq 3}$ belongs to the class of symbol $S^0_{\text{cl}}$ inside the support of our cutoff functions with supremum that is of order $\ml{O}(\varepsilon^2)$. Hence, using the eigenvalue equation~\eqref{e:eigenvalue-conjugation} and the semiclassical a priori estimates~\eqref{e:apriori-semiclassical1} together with the composition rule for pseudodifferential operators and the expressions for $\mathbf{b}$ and $\mathbf{H}_1$ given in~\eqref{e:normal_form_for_b} and~\eqref{e:normalform-H1}, equation~\eqref{e:Wignerequation2} becomes
\begin{multline}\label{e:Wignerequation2bis}
\big \langle\Op_h^w\left(X(b)\left(\lambda_h-W-hQH_1\right)\tilde{\chi}_\varepsilon^C\tilde{\chi}_\varepsilon^{\mathbf{C}}\tilde{\chi}_R^B\tilde{\chi}_R^{\mathbf{B}}\right) u_h,u_h\big\rangle\\[0.2cm] +
\big \langle\Op_h^w\left(\left\{W+hH_1Q,\mathbf{H}_1 \mathbf{b}(\cdot,h \mathbf{H}_1)\right\}\tilde{\chi}_\varepsilon^C\tilde{\chi}_\varepsilon^{\mathbf{C}}\tilde{\chi}_R^B\tilde{\chi}_R^{\mathbf{B}}\right) u_h,u_h\big\rangle=\ml{O}(\varepsilon^2)+\ml{O}_{R,\varepsilon}(h^{\frac{1}{2}}).
\end{multline}
Hence, it remains to analyze the terms
$$\left\{W+hH_1Q,\mathbf{H}_1 \mathbf{b}(\cdot,h \mathbf{H}_1)\right\}.$$
To do that, we recall that, from the exact expressions for $\mathbf{H}_1$ and $\mathbf{b}$ given in~\eqref{e:normalform-H1},~\eqref{e:coeff-a-normal-form} and~\eqref{e:normal_form_for_b}, one has
$$\mathbf{H}_1=H_1\left(1+\sum_{|\alpha|\in\{2,3\}}P_{\alpha}(x,y,z)\left(\frac{H_2}{H_1}\right)^{\alpha_2}\left(\frac{H_3}{H_1}\right)^{\alpha_3}\right),$$
and
\begin{align*}\mathbf{b}(\cdot,h\mathbf{H}_1)\mathbf{H}_1 & =b(\cdot,h\mathbf{H}_1)H_1+X_\perp(b)(\cdot,h\mathbf{H}_1)H_3-V(b)(\cdot,h\mathbf{H}_1)H_2\\[0.2cm]
 & \quad +\sum_{|\alpha|\geq 2}Q_{\alpha}(x,y,z,h\mathbf{H}_1)\frac{H_2^{\alpha_2}H_3^{\alpha_3}}{H_1^{|\alpha|-1}},
\end{align*}
where $Q_\alpha$ are smooth compactly supported functions. We also observe that, in the support of our cutoff functions, one has 
$$\{W+hH_1Q,h\mathbf{H}_1\}=\{W+hH_1Q,hH_1\}+\ml{O}(\varepsilon).$$
For similar reasons, the terms of order $|\alpha|\geq 2$ in the expression of $\mathbf{b}(.,h\mathbf{H}_1)\mathbf{H}_1$ yields a contribution of order $\ml{O}(\varepsilon)$ and~\eqref{e:Wignerequation2bis} can be rewritten as
\begin{multline}\label{e:Wignerequation3}
\big \langle\Op_h^w\left(X(b)\left(\lambda_h-W-hQH_1\right)\tilde{\chi}_\varepsilon^C\tilde{\chi}_\varepsilon^{\mathbf{C}}\tilde{\chi}_R^B\tilde{\chi}_R^{\mathbf{B}}\right) u_h,u_h\big\rangle+\ml{O}(\varepsilon)+\ml{O}_{R,\varepsilon}(h^{\frac{1}{2}})\\[0.2cm] =
-\big \langle\Op_h^w\left(\left\{W+hH_1Q,H_1 b+X_\perp(b)H_3-V(b)H_2\right\}\tilde{\chi}_\varepsilon^C\tilde{\chi}_\varepsilon^{\mathbf{C}}\tilde{\chi}_R^B\tilde{\chi}_R^{\mathbf{B}}\right) u_h,u_h\big\rangle.
\end{multline}
As one has, on the suport of our functions,
\begin{multline*}
\{ W ,  H_1 b+X_\perp(b)H_3-V(b)H_2 \} \\[0.2cm] = - X(W)b - hH_1 X(W) \partial_Eb-X_\perp(b)V(W)+V(b)X_\perp(W)+\ml{O}(\varepsilon)
\end{multline*}
and
\begin{align*}
\{ hH_1 Q,  H_1 b+X_\perp(b)H_3-V(b)H_2 \} & = - hH_1 X(Q) b - (hH_1)^2  X(Q) \partial_Eb + hH_1 Q X(b)\\[0.2cm] & \quad  - hH_1X_\perp(b)V(Q)+hH_1V(b)X_\perp(Q)+\ml{O}(\varepsilon),
\end{align*}
we finally obtain the expected result by letting $h\rightarrow 0^+$, $R\rightarrow+\infty$ and $\varepsilon\rightarrow 0^+$ (in this order) in~\eqref{e:Wignerequation3}.

\section{Summary of the properties of $\mu_\infty$}
\label{s:main_results_and_proofs}

In this short section, we summarize our description of the semiclassical measures $\mu_\infty$ obtained as weak limits of the Wigner distributions $\bm{\mu}_{h}^{R,\varepsilon}$ given by \eqref{e:black_mu} (or equivalently~\eqref{e:full-wigner-H1}). More precisely, as a consequence of Propositions~\ref{p:support_properties} and~\ref{p:invariance}, one has the following Theorem:   

\begin{theo}\label{t:maintheo} Let $Q,W\in\ml{C}^\infty(\ml{M},\IR)$ such that $\|Q\|_{\ml{C}^0}<1$ and let $\lambda_0\geq\min\ W$. Given a sequence $(\psi_h,\lambda_h)$ satisfying \eqref{e:semiclassical-eigenvalue} then any measure $\mu_\infty$ obtained from the sequence~\eqref{e:full-wigner-H1} decomposes as:
\begin{equation}
\label{e:decomposition_mu}
\mu_\infty(q,E) = \overline{\mu}_\infty(q,E) + \sum_{k=0}^\infty  \big( \mu^+_{k,\infty}(q,E) + \mu^-_{k,\infty}(q,E) \big),
\end{equation}
where $\overline{\mu}_\infty$ and $(\mu^\pm_{k,\infty})_{k\geq 0}$ are finite non-negative Radon measures satisfying the following concentration properties:
\begin{enumerate}[(S.1)]

\item
\label{i:support_1} $\operatorname{supp} \overline{\mu}_{\infty} \subset \mathcal{M}_{\lambda_0,W} \times \{ 0 \}$,
with $\mathcal{M}_{\lambda_0,W} := \{ q \in \mathcal{M} \, : \, \lambda_0 - W(q) \geq 0 \}$, 
\medskip

\item
\label{i:support_2} for every $k \in \mathbb{Z}_+$,
$$\operatorname{supp} \mu^\pm_{k,\infty} \subset \mathcal{H}_\pm^{-1}(2k+1) \subset\mathcal{U}_{\lambda_0,W}\times\IR_{\pm}^*.$$
\end{enumerate}
Moreover, they verify the following invariance properties:

\begin{enumerate}[(P.1)]

\item
\label{i_invariance_1} for every $a \in \mathcal{C}^1_c(\mathcal{U}_{\lambda_0,W})$, 
$$
\int_{\ml{M}\times\{0\}}  Y_W \big(a \big) \, d\overline{ \mu}_\infty = 0,
$$ 
with $Y_W$ being defined in~\eqref{e:good-vectorfield},
\medskip

\item 
\label{i:invariance_2} for every $k \in \mathbb{N}$ and every $a \in \mathcal{C}^1_c(\mathcal{M}\times\IR^*)$,
$$
\int_{\ml{M}\times\IR_\pm^*} \frac{\mathcal{X}_{W,Q} (a)}{E}   d\mu^\pm_{k,\infty}  = 0,
$$
with $\mathcal{X}_{W,Q}$ being defined in~\eqref{e:vectorfield-invariance},
\item 
\label{i:invariance_3} for every $a \in \mathcal{C}^1(\mathcal{M})$,
$$
\int_{(\mathcal{M}_{\lambda_0,W}\setminus\mathcal{U}_{\lambda_0,W})\times\{0\}} \big(\Omega_{W}(a)+X(W)a\big)   d\overline{\mu}_{\infty}  = 0.
$$
\end{enumerate}

\end{theo}

\begin{proof}
  The only remaining point compared with Proposition~\ref{p:invariance} is to verify that the invariance properties restrict to each layer $\mathcal{H}_\pm^{-1}(2k+1) $ and $\mathcal{M}\times\{0\}.$ To see this, we first work inside $\mathcal{U}_{\lambda_0,W}\times\IR$ and prove properties \textit{(P.1)} and \textit{(P.2)}. We let $k\in\IZ_+$ and $a\in\ml{C}^{\infty}_{c}(\ml{U}_{\lambda_0,W}\times\IR)$ whose support does not intersect $\text{supp}(\mu_\infty)\setminus\mathcal{H}_\pm^{-1}(2k+1)$. For such a function, we deduce the expected property \textit{(P.2)}. If we now consider $a$ to be an element in $\mathcal{C}^1_c(\ml{M}\times\IR^*)$, then it can be splitted as a sum of a function of the previous form and a function that is supported away from $\mathcal{H}_\pm^{-1}(2k+1)$. Thus, we obtain property \textit{(P.2)} for the expected class of functions. Combining this with Proposition~\ref{p:invariance}, we also find that, for every $a\in\mathcal{C}^1_c(\mathcal{U}_{\lambda_0,W}\times\IR)$,
  $$\int_{\ml{M}\times\{0\}}  Y_W \big((\lambda_0-W)a \big) \, d\overline{ \mu}_\infty = 0,
$$ 
from which we infer \textit{(P.1)}. It now remains to discuss what happens on the critical set 
$$\mathcal{M}_{\lambda_0,W}\setminus\mathcal{U}_{\lambda_0,W}:=\{q\in\mathcal{M}: W(q)=\lambda_0\}.$$ To do this, we rewrite the conclusion of Proposition~\ref{p:invariance} slightly more explicitely: for every $a\in\mathcal{C}^{\infty}_c(\ml{M}\times\IR)$,
$$
\int_{\ml{M}\times\IR}\left((\lambda_0-W)X(a)+\Omega_{\lambda_0-W}(a)-E\Omega_Q(a)+EX(\mathcal{H}_1)\partial_Ea+X(\mathcal{H}_1)a\right)d\mu_{\infty}=0.
$$
If we take $a$ to be of the form $\chi(E/\delta)b(q)$ where $b\in\ml{C}^{\infty}(\mathcal{M})$ and where $\chi$ is the same cutoff function as in the previous sections, we find using the dominated convergence Theorem that, for every $b\in\ml{C}^{\infty}(\ml{M})$, 
$$
\int_{\ml{M}\times\{0\}}\left((\lambda_0-W)X(b)-\Omega_{W}(b)-X(W)b\right)d\overline{\mu}_{\infty}=0.
$$
We now take the test function $b$ to be of the form $\chi((W(q)-\lambda_0)/\delta)\tilde{b}(q)$ with $\chi$ a smooth cutoff function (near $0$) as above. Letting $\delta\rightarrow 0$ in the previous equality, we find thanks to the dominated convergence Theorem
$$
\int_{(\mathcal{M}_{\lambda_0,W}\setminus\mathcal{U}_{\lambda_0,W})\times\{0\}}\big(\Omega_{W}(\tilde{b})+X(W)\tilde{b}\big)d\overline{\mu}_{\infty}=0.
$$
\end{proof}


Note that, compared with Theorem~\ref{t:simplified_theorem} from the introduction, this result holds without any assumption on $\lambda_0$. It also involves the more general measure $\mu_\infty$ which describes precisely how $H_1$ escape at infinity. 
\medskip

Let us now explain that it directly implies Theorem~\ref{t:simplified_theorem}. We also remark that, if $\lambda_0>\max W$, then \textit{(P.1)} reads equivalently as
$\int_{\ml{M}\times\{0\}}  Y_W(a) d\overline{ \mu}_\infty = 0$ for every $a\in\ml{C}^1(\ml{M})$. This implies the property of $\overline{\nu}_\infty$ in Theorem~\ref{t:simplified_theorem} by letting
$$\overline{\nu}_\infty(q) = \int_\IR \overline{\mu}_\infty(q,dE).$$

Notice, since $\Omega_{\mathcal{H}_1}(\mathcal{H}_1) = 0$, that $\mathcal{X}_{W,Q}(\mathcal{H}_\pm) = 0$. This implies in particular that the vector field $E^{-1}\mathcal{X}_{W,Q}$ is tangent to the level sets $\mathcal{H}_\pm^{-1}(2k+1)$ and thus induces a well-defined flow on these layers. Finally, we can derive from~\textit{(S.2)} and \textit{(P.2)} that
$$
\int_{\ml{M}\times\IR^*_{\pm}}\left( \left(\pm(2k+1)+Q\right)Y_W(a)-\Omega_Q(a)+X(\ml{H}_1)\partial_Ea  \right) d\mu^\pm_{k,\infty}  = 0,
$$
which implies the last part of Theorem~\ref{t:simplified_theorem} by letting
$$\nu_{k,\infty}^\pm(q) := \int_\IR \mu_{k,\infty}^\pm(q,dE).$$

\section{The case of the flat torus}\label{s:torus}

In this section, we briefly discuss the case where $M= \mathbb{T}^2 = \IR^2/2\pi\IZ^2$, $Q=W=0$ and $g=dx^2+dy^2$ is the canonical Euclidean metric. Our aim is to show examples of different sequences of eigenfunctions for the operator $-h^2 \Delta_{\operatorname{sR}}$ which select any given choice among the semiclassical measures $\overline{\mu}_\infty$ and $\mu_{k,\infty}^{\pm}$ by putting their total mass on them as $h \to 0^+$. 
\medskip

In this particular example, the operators $X$, $X_\perp$ and $V$ can be written by global formulas in the canonical coordinates $(x,y,z) \in \mathbb{T}^3 \simeq S\mathbb{T}^2$. Precisely:
$$X=\cos z\partial_x+\sin z \partial_y,\ X_\perp=\sin z\partial_x-\cos z\partial_y,\ \text{and}\ V=\partial_z.$$ 
We restrict ourselves to search for solutions to~\eqref{e:semiclassical-eigenvalue} of the particular form 
$$\psi_h(x,y,z)=u_h(z)e^{in \cdot (x,y)},\quad \mathbf{n}=(n_1,n_2)\in\IZ^2.$$
As we impose that $\psi_h$ solves our eigenvalue problem, then $u_h$ must satisfy
$$h^2(n_1\sin z-n_2\cos z)^2u_h(z)-h^2u_h''(z)=u_h(z),$$
or equivalently
$$-\frac{1}{\|\mathbf{n}\|^2}u_h''(z)+\sin^2(z-z_\mathbf{n})u_h(z)=\frac{1}{h^2\|\mathbf{n}\|^2}u_h(z),$$
with $(\cos z_\mathbf{n},\sin z_\mathbf{n})=\mathbf{n}/\|\mathbf{n}\|$. We recognize in this expression the semiclassical Mathieu operator 
$$
\widehat{M}_\mathbf{n} := - \frac{1}{\Vert \mathbf{n} \Vert^2} \partial_z^2 + \sin^2(z- z_\mathbf{n}),
$$
on the circle $\mathbb{S}^1$. Note that it is a one-dimensional Schr\"odinger operator with a double well potential. Hence, we are led to the equation
\begin{equation}
\label{e:Mathieu_equation}
\widehat{M}_\mathbf{n} u_\mathbf{n} = \lambda(\mathbf{n}) u_\mathbf{n}, \quad \lambda(\mathbf{n}) = \frac{1}{(h_\mathbf{n} \Vert \mathbf{n} \Vert)^2}.
\end{equation}
Since $\widehat{M}_\mathbf{n}$ has compact resolvent on $L^2(\mathbb{S}^1)$, for every $\mathbf{n}\in\IZ^2$ there is an increasing sequence of eigenvalues $(\lambda_k(\mathbf{n}))_{k\geq 0}$ with corresponding (normalized) eigenfunctions $(u_{\mathbf{n},k}(z))_{k\geq 0}$ of $\widehat{M}_\mathbf{n}$. Note that, up to translation by $z_{\mathbf{n}}$, we can restrict ourselves to the case where $z_\mathbf{n}=0$ which amounts to take a lattice point of the form $\mathbf{n}=(n,0)$ with say $n>0$. Under this assumption, we first show the following:
\begin{lemm}
\label{l:quasimodes_for_mathieu}
For every $k \in \mathbb{N}$, there exists $\lambda_k(n,0):=\lambda_k(n) \in \operatorname{Sp}_{L^2(\mathbb{S}^1)}(\widehat{M}_{(n,0)})$ such that:
\begin{equation}
\label{e:eigenvalues_asymptic}
\lambda_k(n) = \frac{(2k+1)}{ n } \left( 1 + \ml{O}_k\left(\frac{1}{\sqrt{n}}\right) \right), \quad \text{as }  n \to + \infty.
\end{equation}
\end{lemm}

\begin{rema} In principle, there could be other sequences of eigenvalues verifying different asymptotic formulas, say $\lambda_\alpha(n)=\frac{\alpha}{n}+o(n^{-1}).$ Yet, one could show that this is not the case by comparing eigenfunctions of $\widehat{M}_n$ with quasimodes of the harmonic oscillator 
\begin{equation}
\label{e:harmonic_oscillator}
\widehat{H}_n := - \frac{1}{ n^2} \partial_z^2 + z^2,
\end{equation}
on $L^2(\IR)$, whose spectrum is given explicitely by
$$
\operatorname{Sp}_{L^2(\IR)}(\widehat{H}_n) = \left \{ \frac{2k+1}{n } \, : \, k \in \mathbb{N} \right \}.
$$ 
However, since we are only interested in showing the existence of sequences of eigenfunctions of $-h_n^2 \Delta_{\operatorname{sR}}$ which put positive mass on the semiclassical measures $\overline{\mu}_\infty$ and $\mu_{k,\infty}^\pm$, for which we already know the concentration properties \textit{(S.1)} and \textit{(S.2)} of Theorem \ref{t:maintheo}, we omit this discussion.
\end{rema}

\begin{proof}
To prove this lemma, it is sufficient to construct a sequence of (almost) normalized quasimodes $(v_{k,n}, E_k(n))_{n\rightarrow \infty}$ satisfying
$$
\widehat{M}_n v_{k,n} = E_k(n) v_{k,n} + \ml{O}_k\left( \frac{1}{n^{\frac{3}{2}}} \right); \quad E_k(n) = \frac{2k+1}{ n}.
$$
To this aim, let $\delta > 0$ small, and set $v_{k,n}(z) := \chi(z/\delta)\varphi_{k,n}$ where $\chi$ is a cutoff function supported in a neighborhood of $0$ and where 
$$
\varphi_{k,n}(z) = n^{1/4} \varphi_k\big( \sqrt{n}z \big),
$$
and $\varphi_k$ is the normalized Hermite function of degree $k$~\cite[Th.~6.2]{Zworski12}. Observe that, for any $N \geq 1$ and for any $\ell\geq 0$,
\begin{equation}\label{e:exp-decay}
\int_{|z|\geq \delta}\left|\varphi_{k,n}^{(\ell)}(z)\right|^2dz = \mathcal{O}_{\delta,\ell,k,N}\left( \frac{1}{ n^{N/2}} \right), \quad \text{as } \|n\| \to \infty.
\end{equation}
In particular, we get that $\|v_{k,n}\|_{L^2(\IT)}=1+\ml{O}_k(n^{-1})$ as $n\to\infty$ as expected. Moreover, using the Taylor expansion
$$
\sin^2z = z^2 + \mathcal{O}(\vert z \vert^3), \quad \text{as }  z \to 0,
$$
and the fact that
$$
\left(\int_{\IR}|z^3\varphi_{k,n}(z)|^2 dz\right)^{\frac{1}{2}} = \mathcal{O}_k \left( \frac{1}{n^{3/2}} \right), \quad \text{as } n\rightarrow\infty,
$$
the claim holds by using the eigenvalue equation combined with~\eqref{e:exp-decay} and
$$
\widehat{H}_n \varphi_{k,n} = E_{k}(n) \varphi_{k,n}, \quad z \in \IR.
$$
\end{proof}

Let us now fix, for every $n>0$,
\begin{equation}
\label{e:choice_h_n}
h_n := \frac{1}{\sqrt{(2k+1 + o_k(1)) n }}, \quad \text{as }  n  \to \infty,
\end{equation}
so that \eqref{e:Mathieu_equation} and \eqref{e:eigenvalues_asymptic} hold. For this sequence, take a sequence of solutions to \eqref{e:semiclassical-eigenvalue} that are of the form
\begin{equation}
\label{e:eigenfunction_k}
\psi_n(x,y,z) = u_{n}(z)e^{inx}.
\end{equation}

\begin{prop}
\label{p:examples_sequences}
Let $(\psi_n)_{n\geq 1}$ be a normalized sequence of the form \eqref{e:eigenfunction_k}. Let us assume that $(\psi_n)_{n\geq 1}$ satisfy \eqref{e:semiclassical-eigenvalue} with $(h_n)_{n\geq 1}$ given by \eqref{e:choice_h_n}. Then the total mass of $\mu^+_{k,\infty} + \mu^-_{k,\infty}$ is equal to one. 
\end{prop}

\begin{proof}
Let us consider $\delta > 0$ to be chosen sufficiently small along the proof. Let $\chi_\delta = \chi(\cdot/\delta)$ where $\chi$ is still a small cutoff function near $0$. We have, by the functional analysis of pesudodifferential operators \cite[Thm. 14.9]{Zworski12}, the localization property:
\begin{align*}
\Op_{\frac{1}{\Vert n \Vert}}^{\mathbb{S}^1,w} \big( \chi_\delta(  \sin^2(z) + \zeta^2 ) \big)   u_n & =  \chi_{\delta}(  \widehat{M}_n)  u_n + O_\delta\left( \frac{1}{ n} \right) \\[0.2cm]
 & = \chi_\delta\left(  \frac{2k+1}{n}(1+o_k(1))\right) u_n + O_\delta\left( \frac{1}{n } \right) \\[0.2cm]
 & = u_n + O_{\delta,k}\left( \frac{1}{ n } \right).
\end{align*}
On the other hand, we have:
\begin{align*}
\big \langle \Op_{h_n}^w( \chi(h_n H_1) \tilde{\chi}_\varepsilon^C \tilde{\chi}_R^B ) \psi_n, \psi_n \big \rangle_{L^2(\IT^3)} = \big \langle \Op_{h_n}^{\mathbb{S}^1,w} (\kappa_n^{\varepsilon,R}) u_n,u_n \big \rangle_{L^2(\mathbb{S}^1)},
\end{align*}
where
$$
\kappa_n^{\varepsilon,R}(z,\zeta) := \chi(h_n^2 n \cos(z) )   \tilde{\chi} \left( \frac{ \varepsilon h_n n \cos(z) }{\sqrt{1 +  (h_n n)^2\sin^2(z) + \zeta^2}} \right) \tilde{\chi} \left( \frac{ (h_n n )^2 + \zeta^2}{R} \right).
$$
Arguing as in Section~\ref{s:cutoff}, one finds that this symbol belongs to the class of symbols $S^0_{\text{cl}}(T^*\IS^1)$ amenable to pseudodifferential calculus on the circle. Observe also that, for $n$ large enough, the last function in this product is identically equal to $1$ Notice also that
\begin{align*}
\Op_{\frac{1}{\Vert n \Vert}}^{\mathbb{S}^1,w} \big( \chi_\delta(  (\sin^2(z) + \zeta^2 )) \big)   & = \Op_{h_n}^{\mathbb{S}^1,w} \left( \chi_\delta\left(  \sin^2(z) + \left( \frac{\zeta}{h_n  n} \right)^2 \right) \right) \\[0.2cm]
& =: \Op_{h_n}^{\mathbb{S}^1,w} ( \sigma_n^\delta ).
\end{align*}
Thus, by using the previous localization property for $u_n$ and the semiclassical pseudodifferential calculus, we have the composition formula:
\begin{align*}
 \big \langle \Op_{h_n}^{\mathbb{S}^1,w} (\kappa_n^{\varepsilon,R}) u_n,u_n \big \rangle_{L^2(\mathbb{S}^1)} & =  \big \langle \Op_{h_n}^{\mathbb{S}^1,w} (\kappa_n^{\varepsilon,R}) \Op_{h_n}^{\mathbb{S}^1,w} \big( \sigma_n^\delta \big) u_n,u_n \big \rangle_{L^2(\mathbb{S}^1)} + \mathcal{O}_{\delta,k}\left( \frac{1}{ n } \right) \\[0.2cm]
  & =  \big \langle \Op_{h_n}^{\mathbb{S}^1,w} (\kappa_n^{\varepsilon,R}  \sigma_n^\delta \big) u_n,u_n \big \rangle_{L^2(\mathbb{S}^1)} + \mathcal{O}_{\delta,\varepsilon,R}\left( h_n \right).
\end{align*}
In this expression, if we take $\delta$ sufficiently small (i.e. so that $\delta \ll\varepsilon^2$), we have 
$$
\kappa_n^{\varepsilon,R}(z,\zeta)  \sigma_n^\delta(z,\zeta) =   \chi(h_n^2  n  \cos(z)) \sigma_n^\delta(z,\zeta),
$$ 
since $\tilde{\chi}(x) = 1$ for $\vert x \vert \geq 2$. Moreover, for $\delta > 0$ sufficiently small, we can also decompose $\sigma_n^\delta$ as the sum of two functions $\sigma_n^{1,\delta}$ and $\sigma_n^{2,\delta}$ compactly supported respectively near $z = 0$ and $z = \pi$, that is:
$$
\sigma_n^\delta(z,\zeta) = \sigma_n^{1,\delta}(z,\zeta) + \sigma_n^{2,\delta}(z,\zeta),
$$
with $\supp \sigma_n^{1,\delta} \cap \supp \sigma_n^{1,\delta} = \emptyset$. Using next that
$$
h_n^2 n  \cos(z) = \frac{(-1)^{j-1}}{2k+1} + \mathcal{O}(\delta), \quad \text{as } \delta \to 0, \quad j = 1,2,
$$
respectively on the support of $\sigma_n^{j,\delta}(z,\zeta)$, we get thanks to the Calder\'on-Vaillancourt Theorem:
\begin{align*}
 \big \langle \Op_{h_n}^{\mathbb{S}^1,w} (  \kappa_n^{\varepsilon,R}) u_n,u_n \big \rangle_{L^2(\mathbb{S}^1)} & \\[0.2cm]
  & \hspace*{-3cm} = \sum_{j\in \{1,2 \}}  \chi\left( \frac{(-1)^{j-1}}{2k+1} \right) \big \langle \Op_{h_n}^{\mathbb{S}^1,w} (  \sigma_n^{j,\delta} ) u_n,u_n \big \rangle_{L^2(\mathbb{S}^1)} + \mathcal{O}(\delta) +  \mathcal{O}_{\delta,\varepsilon,R}\left( h_n \right).
\end{align*}
Therefore, taking limits in $ n  \to + \infty$ through a subsequence, we obtain, for $\delta \ll \varepsilon^2$,
$$
 \lim_{h_n \to 0}\big \langle \Op_{h_n}^w( \chi(h_n H_1) \tilde{\chi}_\varepsilon^C \tilde{\chi}_R^B ) \psi_n, \psi_n \big \rangle_{L^2(\mathcal{M})} =  \alpha_1^\delta \chi \left( \frac{1}{2k+1} \right) + \alpha^\delta_2\chi\left(  \frac{-1}{2k+1} \right) + \mathcal{O}(\delta),
$$
where $\alpha_j^\delta = \lim_{h_n \to 0^+} \big \langle \Op_{h_n}^{\mathbb{T}^1,w} (  \sigma_n^{j,\delta} ) u_n,u_n \big \rangle_{L^2(\mathbb{T}^1)}$, and $\alpha_1^\delta + \alpha_2^\delta = 1$ (using one more time the localization property of the sequence $(u_n)_{n\geq 1}$). Finally, in view of the fact that
$$
\lim_{\varepsilon \to 0} \lim_{R \to \infty} \lim_{h_n \to 0}\big \langle \Op_{h_n}^w( \chi(h_n H_1) \tilde{\chi}_\varepsilon^C \tilde{\chi}_R^B ) \psi_n, \psi_n \big \rangle_{L^2(\IT^3)} = \int_{\IT^3\times \IR} \chi(E) \, d\mu_\infty(q,E),
$$ 
we can take $\delta \to 0$ and use Theorem \ref{t:maintheo} (property \textit{(S.2)}) to conclude the proof.
\end{proof}

We finally show the existence of sequences of eigenfunctions $(\psi_n)$ satisfying \eqref{e:semiclassical-eigenvalue} which put positive mass on the semiclassical measure $\overline{\mu}_\infty$. To this aim, we note that, thanks to Lemma~\ref{l:quasimodes_for_mathieu} and for every $k\geq 1$, we can find some $n_k\geq 1$ such that, for every $n\geq n_k$, there is an eigenvalue $E_k(n)$ of $\widehat{M}_{(n,0)}$ verifying
$$\frac{1}{2\sqrt{(2k+1)n}}\leq E_k(n)=\frac{1}{(h_nn)^2}\leq\frac{\sqrt{2}}{\sqrt{(2k+1)n}}\leq \frac{1}{\sqrt{k}}.$$
Hence, we can take $n=n_k$ and pick a sequence $(K_{n_k})_{k\geq 1}$ such that $K_{n_k} \to + \infty$ and thus the sequence $(h_{n_k})_{k\geq 1}$ satisfying now
\begin{equation}
\label{e:choice_h_n_2}
\frac{1}{n_k}\ll h_{n_k} = \frac{1}{\sqrt{K_{n_k} n_k }}\ll \frac{1}{\sqrt{n_k}}, \quad \text{as }  k \to \infty.
 \end{equation}
Adapting the proof of Proposition \ref{p:examples_sequences} and using property \textit{(S.1)} of Theorem \ref{t:maintheo}, we obtain:
\begin{coro}
\label{c:examples_sequences}
 Let $(\psi_{n_k})_{k\geq 1}$ be a normalized sequence of the form \eqref{e:eigenfunction_k}. Let us assume that $\psi_{n_k}$ satisfy \eqref{e:semiclassical-eigenvalue} with $h_{n_k}$ given by \eqref{e:choice_h_n_2}. Then the total mass of $\overline{\mu}_\infty$ is equal to one.
\end{coro}

\appendix

\section{Spectral properties of $\widehat{P}_h$}
\label{a:spectral}

In this appendix, we briefly review the spectral properties of $\widehat{P}_h$. A key ingredient of the analysis is the following standard result~\cite[Cor.~17.14]{RothschildStein76}:
\begin{theo}[Rothschild-Stein]\label{t:RothschildStein} Let $Q\in\ml{C}^\infty(\ml{M},\IR)$ such that $\|Q\|_{\ml{C}^0}<1$. Set 
$$\mathcal{L}=-\Delta_{\text{sR}}-iQX=-X_\perp^2-V^2-iQ[V,X_\perp].$$

Then, for every $N\geq 1$, one can find continuous maps $P_N:H^s\rightarrow H^{1+s}$ and $S_N:H^s\rightarrow H^{s+\frac{N}{2}}$ (for all $s\geq 0$) such that 
$$P_N\mathcal{L}=\operatorname{Id}+S_N.$$ 

In particular, there exists a constant $C_{M,g}>0$ such that
\begin{equation}\label{e:RothschildStein}
\forall \psi\in\mathcal{C}^\infty(\ml{M}),\quad\|\psi\|_{H^1}\leq C_{M,g}\left(\|\mathcal{L}\psi\|_{L^2}+\|\psi\|_{L^2}\right),
 \end{equation} 
and, for every $s\geq 0$,
$$\mathcal{L}(\psi)\in H^s,\quad\psi\in L^2\ \Longrightarrow\ \psi\in H^{s+1}.$$
\end{theo}

Let us now discuss the spectral properties of $\widehat{P}_h$. For an introduction on the spectral properties of unbounded operators, the reader is referred to~\cite[Ch.~VIII]{ReedSimon80} and~\cite[Ch.~X]{ReedSimon75} that we closely follow for the terminology. For any $\psi\in H^2(\ml{M})$, we define
$$\tilde{P}_h(\psi):=\left(-h^2\Delta_{\text{sR}}+\frac{h^2}{2i}(QX-(QX)^*)+W\right)\psi,$$
which induces an unbounded operator
$$\tilde{P}_h:D(\tilde{P}_h):=H^2(\mathcal{M})\subset L^2(\ml{M})\rightarrow  L^2(\ml{M}).$$
One can define its adjoint $\tilde{P}_h^*$ by defining the domain
$$D(\tilde{P}_h^*):=\left\{\psi\in L^2(\ml{M}):\ \exists u\in L^2(\ml{M})\ \text{such that}\ \forall \varphi\in H^2(\ml{M}),\ \langle\psi,\tilde{P}_h\varphi\rangle=\langle u,\varphi\rangle\right\},$$
or equivalently
$$D(\tilde{P}_h^*):=\left\{\psi\in L^2(\ml{M}):\ \tilde{P}_h\psi\in L^2(\ml{M})\right\}$$
The operator $\tilde{P}_h^*:\psi\in D(\tilde{P}_h^*)\rightarrow \tilde{P}_h\psi \in L^2(\ml{M})$ is closed and it is densely defined. Hence, according to~\cite[Th.~VIII.1]{ReedSimon80}, $\tilde{P}_h$ is closable and we denote its closure by $\overline{\tilde{P}}_h$ whose domain is denoted by $D\left(\overline{\tilde{P}}_h\right)$ and equal to the set of $\psi\in L^2(\ml{M})$ such that
$$\exists \psi_j\in H^2(\ml{M}),\ \exists v\in L^2(\ml{M})\ \text{such that}\  \|\psi_j-\psi\|_{L^2}+\|\tilde{P}_h\psi_j-v\|_{L^2}\rightarrow 0.
$$
In general, one only has $D\left(\overline{\tilde{P}}_h^*\right)= D(\tilde{P}_h^*)\subset D\left(\overline{\tilde{P}_h}\right)$ so that $\overline{\tilde{P}}_h$ is not necessarily selfadjoint. In order to fix this problem, we can make some assumptions on the size of $Q$ and use positivity arguments.

More precisely, $\tilde{P}_{h}$ is associated with the real quadratic form
$$\tilde{B}(\psi):=\int_{\ml{M}}(\tilde{P}_h\psi)\overline{\psi} \, d\mu_L,\ \psi\in H^2(\ml{M}),$$
which, thanks to~\eqref{e:commutators}, is bounded from below by
\begin{align*}\tilde{B}(\psi) & \geq\|hX_\perp\psi\|^2_{L^2} +\|hV\psi\|^2_{L^2}-2\|Q\|_{\ml{C}^0}\|hX_\perp\psi\|_{L^2}\|hV\psi\|_{L^2} \\[0.2cm]
 & \quad +(\min W-h^2\|Q\|_{\ml{C}^1})\|\psi\|^2_{L^2}\\[0.2cm]
& \quad -h\|Q\|_{\ml{C}^1}\|\psi\|_{L^2}\big(\|hX_\perp\psi\|_{L^2}+\|hV\psi\|_{L^2}\big)\\[0.2cm]
& \geq \left(1-\|Q\|_{\ml{C}^0}-\frac{h\|Q\|_{\ml{C}^1}}{2}\right) \left(\|hX_\perp\psi\|^2_{L^2}+\|hV\psi\|^2_{L^2}\right)\\[0.2cm]
& \quad +\left(\min W-(h^2+h)\|Q\|_{\ml{C}^1}\right)\|\psi\|^2_{L^2}.
\end{align*}
Hence, if $\|Q\|_{\ml{C}^0}<1$ (and $h>0$ is small enough in a way that depends on $Q$), it follows from~\cite[Th.~X.23]{ReedSimon75} that $\tilde{B}$ is a closable form whose closure $B$ corresponds to a unique selfadjoint operator $\widehat{P}_h$ referred as the \emph{Friedrichs extension} of $\tilde{P}_h$. Moreover, the spectrum of this selfadjoint extension is bounded from below by $\min W+\ml{O}_Q(h)$ and its domain verifies
$$D(\widehat{P}_h)\subset H^1_{\text{sR}}(\ml{M}):=\left\{u\in\ml{D}'(\ml{M}): \|\psi\|_{L^2}^2+\|X_\perp\psi\|_{L^2}^2+\|V\psi\|^2_{L^2}<\infty\right\}.$$
In particular, $\widehat{P}_h:D(\widehat{P}_h)\subset L^2(\ml{M})\rightarrow L^2(\ml{M})$ is a closed selfadjoint operator and thus $(\widehat{P}_h+C)$ has a bounded inverse for $C>0$ large enough:
$$\big(\widehat{P}_h+C\big)^{-1}:L^2(\ml{M})\rightarrow (D(\widehat{P}_h),\|.\|_{L^2})\subset L^2(\ml{M}).$$
We would like to show that this defines a compact operator. To see this, recall from~\eqref{e:RothschildStein} that
$$\forall\psi\in\ml{C}^{\infty}(\ml{M}),\quad \|(\widehat{P}_h+C)^{-1}\psi\|_{H^1(\ml{M})}\leq c_h\big(\|\psi\|_{L^2}+\|(\widehat{P}_h+C)^{-1}\psi\|_{L^2}\big)$$
so that, if $(\psi_j)_{j\geq 0}$ is a bounded sequence in $L^2(\ml{M})$, then $((\widehat{P}_h+C)^{-1}\psi_j)_{j\geq 0}$ is also bounded in $H^1(\ml{M})$. 
\begin{rema}\label{r:domain} Along the way, this discussion shows that $H^2(\ml{M})\subset D(\widehat{P}_h)\subset H^1(\ml{M})$ (with continuous inclusions).
\end{rema}

As the inclusion $H^1(\ml{M})\subset L^2(\ml{M})$ is compact, $(\widehat{P}_h+C)^{-1}:L^2(\ml{M})\rightarrow L^2(\ml{M})$ is indeed a compact operator. As $\widehat{P}_h$ is selfadjoint, there exists an orthonormal basis of $L^2(\ml{M})$ made of eigenmodes of $\widehat{P}_h$. Moreover, if one has $\widehat{P}_h\psi_h=\lambda_h\psi_h$ with $\psi_h\in D(\widehat{P}_h)$, then $\mathcal{L}\psi_h\in H^1(\ml{M})$ and, according to Theorem~\ref{t:RothschildStein}, one finds that $\psi_h\in H^2(\ml{M})$. By induction, we get that these eigenmodes are smooth.

 \begin{rema} If we let $C>0$ be a large enough constant, then $\tilde{P}_h+C$ is a positive symmetric operator and its adjoint is given by $\tilde{P}_h^*+C$ with domain $D(\tilde{P}_h^*)$. In particular, if $\psi$ belongs to the kernel of $\tilde{P}_h^*+C$, then, by the Rothschild-Stein Theorem, $\psi$ belongs to $H^1(\mathcal{M})$ (and by induction to $\mathcal{C}^{\infty}(\ml{M})$. Hence, it lies in the domain of $\widehat{P}_h$ and we can deduce that $\psi=0$. According to~\cite[Th.~X.26]{ReedSimon75}, it implies that the Friedrichs extension is the only (semibounded) selfadjoint extension of $\tilde{P}_h+C$ (hence of $\tilde{P}_h$). 
 \end{rema}

The spectral properties of $\widehat{P}_h$ that we have proved so far are summarized by the next statement:
\begin{lemm}\label{l:spectrum}  Suppose that $\|Q\|_{\ml{C}^0}<1$. Then, there exists $h_0>0$ such that, for every $0<h<h_0$, 
$$\widehat{P}_h:D(\widehat{P}_h)\rightarrow L^2(\ml{M})$$
is a selfadjoint operator whose spectrum consists in a discrete sequence of eigenvalues
$$\min W+\ml{O}_Q(h)\leq \lambda_h(0)\leq\lambda_h(1)\leq\ldots\leq \lambda_h(j)\ldots\rightarrow +\infty.$$
Moreover,
$$\widehat{P}_h\psi_h=\lambda_h\psi_h,\ \text{with}\ \psi_h\in D(-\widehat{P}_h)\quad\Longrightarrow\quad \psi_h\in \ml{C}^\infty(\ml{M}).$$
\end{lemm}

We conclude this appendix with the following a priori estimates that are used all along the article:
\begin{lemm}\label{l:apriori-estimate} Suppose that $\|Q\|_{\ml{C}^0}<1$. Then, one can find $C_{Q,W}>0$ and $0<h_Q\leq 1$, such that, for all $0<h\leq h_Q$,
 \begin{multline*}\widehat{P}_h\psi_h=\lambda_h\psi_h,\ \text{with}\ \psi_h\in D(\widehat{P}_h)\\
 \Longrightarrow\quad \|hX_\perp\psi_h\|_{L^2}^2+\|hV\psi_h\|_{L^2}^2+\|h^2X\psi_h\|_{L^2}^2\leq C_{Y,W}(1+|\lambda_h|)^2\|\psi_h\|_{L^2}^2.
\end{multline*}
\end{lemm}
\begin{proof} Let $\psi_h\in D(\widehat{P}_h)$ such that $\widehat{P}_h=\lambda_h\psi_h$. One has then
$$\|hX_\perp\psi_h\|_{L^2}^2+\|hV\psi_h\|_{L^2}^2=\lambda_h\|\psi_h\|_{L^2}^2-\langle W\psi_h,\psi_h\rangle-\frac{h^2}{i}\langle(QX+\frac{1}{2}X(Q))\psi_h,\psi_h\rangle.$$
Hence, one has
$$\|hX_\perp\psi_h\|_{L^2}^2+\|hV\psi_h\|_{L^2}^2\leq (\|W\|_{\ml{C}^0}+|\lambda_h|)\|\psi_h\|_{L^2}^2+h^2|\langle(QX+\frac{1}{2}X(Q))\psi_h,\psi_h\rangle|.$$
Recall that $X=[V,X_\perp]$ from which we infer
\begin{multline*}|\langle(QX+\frac{1}{2}X(Q))\psi_h,\psi_h\rangle|\leq 2\|Q\|_{\ml{C}^0}\|hX_\perp\psi_h\|_{L^2}\|hV\psi_h\|_{L^2}\\
+ h\|Q\|_{\ml{C}^1}\left(\|hX_\perp\psi_h\|_{L^2}+ \|hV\psi_h\|_{L^2}\right)\|\psi_h\|_{L^2}
+ \frac{h^2}{2}\|Q\|_{\ml{C}^1}\|\psi_h\|_{L^2}^2.\end{multline*}
Then, we get
$$\|hX_\perp\psi_h\|_{L^2}^2+\|hV\psi_h\|_{L^2}^2\leq \frac{1}{1-\|Q\|_{\ml{C}^0}-\frac{h\|Q\|_{\ml{C}^1}}{2}}(\|W\|_{\ml{C}^0}+|\lambda_h|+2h\|Q\|_{\ml{C}^1})\|\psi_h\|_{L^2}^2.$$
Hence, under the assumption that $\|Q\|_{\ml{C}^0}$, there exists a constant $C_{Q,W}>0$ (depending only $Q$ and $W$) and $0<h_Q\leq 1$ (depending only on $Q$) such that, for every $0<h\leq h_Q$,
\begin{equation}\label{e:bound-apriori-proof}\|hX_\perp\psi_h\|_{L^2}^2+\|hV\psi_h\|_{L^2}^2\leq C_{Q,W}(1+|\lambda_h|)\|\psi_h\|_{L^2}^2.\end{equation}
Finally, using the Rothschild-Stein Theorem one more time, one finds that there exists a constant $C_{M,g}>0$ such that
$$\|X\psi_h\|_{L^2}\leq \|\psi_h \|_{H^1}\leq C_{M,g}\left(\|\mathcal{L}\psi_h\|_{L^2}+\|\psi_h\|_{L^2}\right).$$
Multiplying this inequality by $h^2$ and using the fact that $\widehat{P}_h\psi_h=\lambda_h\psi_h$ to control the upper bound in terms of $\|\psi_h\|_{L^2}$, we obtain the expected upper bound.
\end{proof}

\section{Reminder on semiclassical analysis on $\IR^2\times\IS^1$}\label{a:pdo}

In this appendix, we review a few facts about semiclassical analysis on $T^*(\IR^2\times\IS^1)$ that are used all along our analysis of the measure at infinity. A standard textbook is~\cite{Zworski12} which treats the case of $T^*\IR^3$ in great details in Chapter~$4$. The case of $T^*(\IR^2\times\IS^1)$ can be handled similarly by proper use of Fourier series along the $z$-variable rather than Fourier transform. See for instance~\cite[\S 5.3]{Zworski12} for a detailed discussion in the case of $T^*\IT^3$.

For a nice enough smooth function $a$ on $T^*(\IR^2\times\IS^1)$ (say compactly supported) and for every $h>0$, the Weyl (semiclassical) quantization of $a$ is defined, for all $u$ in $u\in\ml{C}^\infty_c(\IR^3)$, by
\begin{equation}\label{e:weyl} \text{Op}_h^w(a)\left(u\right)(q):=\frac{1}{(2\pi h)^3}\int_{\IR^6}e^{\frac{i}{h}(q-q')\cdot p}a\left( \frac{q+q'}{2},p\right)u(q')dq'dp.
\end{equation}
Using the periodicity along the $\mathbb{S}^1$-variable, one can verify that this definition extends to smooth test functions $u\in\ml{C}^\infty_c(\IR^2\times\IS^1)$~\cite[\S5.3.1]{Zworski12}.

Regarding the regularity needed for $a$, this definition still makes sense when working with smooth functions $a$ belonging to the class of (Kohn-Nirenberg) symbols~\cite[\S 9.3]{Zworski12}:
$$S^m_{\text{cl}}(T^*(\IR^2\times\IS^1))=\left\{a\in\ml{C}^\infty(T^*(\IR^2\times\IS^1)):\ \forall(\alpha,\beta)\in\IZ_+^{6},\ P_{m,\alpha,\beta}(a)<+\infty\right\},\quad m\in\IR,$$
where
$$P_{m,\alpha,\beta}(a):=\sup_{(q,p)}\{\langle p\rangle^{-m+|\beta|}|\partial_q^\alpha\partial_p^\beta a(x,\xi)|\}.$$
In other words, we gain some decay in $p$ when differentiating in the $p$-variable. Even if such a decay is not necessary to work in an Euclidean set-up, it is of crucial importance in our analysis to have this extra decay in view of dealing with the escape at infinity in the fibers.

A nice property of the Weyl quantization is that, for a real-valued $a$, $\Op_h^w(a)$ is a (formally) selfadjoint operator~\cite[Th.~4.1]{Zworski12}. Another property that we extensively use all along this article is the composition rule for pseudodifferential operators\footnote{Technically speaking, this reference deals with the Weyl quantization on $T^*\IR^3$ but the proof works as well in our set-up.}~\cite[Th.~9.5, Th.~4.12]{Zworski12}
\begin{theo}\label{t:composition} Let $a\in S^{m_1}_{\text{cl}}(T^*(\IR^2\times\IS^1))$ and $b\in S^{m_2}_{\text{cl}}(T^*(\IR^2\times\IS^1))$. Then, there exists $c\in S^{m_1+m_2}_{\text{cl}}(T^*(\IR^2\times\IS^1))$ (depending on $h$) such that
\begin{equation}\label{e:composition-formula}\Op_h^w(a)\circ\Op_h^w(b)=\Op_h^w(c).\end{equation}
Moreover,
$$c(q,p)=\sum_{k=0}^N\frac{h^k}{k!}\left(A(D)\right)^k(a(q_1,p_1)b(q_2,p_2))|_{q_1=q_2=q, p_1=p_2=p}+\ml{O}_{S^{m_1+m_2-N-1}}(h^{N+1}),$$
 where the constant in the remainder depends on a finite number of seminorms of $a$ and $b$ (depending on $N$ and on the seminorm in $S^{m_1+m_2-N-1}$), and where 
 $$A(D):=\frac{1}{2i}\left(\partial_{p_1}\cdot\partial_{q_2}-\partial_{p_2}\cdot\partial_{q_1}\right).$$
\end{theo}
In particular, we can see from this result that $c=\ml{O}_{S^{m_1+m_2-N-1}}(h^{N+1})$ if $a$ and $b$ have disjoint supports. We can also verify that, all the even powers in $h$ in the asymptotic expansion of $[\Op_h^w(a),\Op_h^w(b)]$ cancels out and that the first term is given by $\frac{h}{i}\{a,b\}.$

Another key property for us is the Calder\'on-Vaillancourt Theorem~\cite[Ch.~5]{Zworski12} that states the existence of constants $C_0, N_0$ such that, for every $a\in S^{0}_{\text{cl}}(T^*(\IR^2\times\IS^1))$,
\begin{equation}\label{e:calderon}
 \left\|\Op_h^w(a)\right\|_{L^2\rightarrow L^2}\leq C_0\sum_{|\alpha|\leq N_0}h^{\frac{|\alpha|}{2}}\|\partial^\alpha a\|_{\infty}.
\end{equation}
Recall also the Garding property that is valid for elements in $S^{0}_{\text{cl}}(T^*(\IR^2\times\IS^1))$. Given any $a$ in that class satisfying $a\geq 0$, it ensures the existence of a constant $C_a>0$~\cite[Th.~4.32]{Zworski12} such that
\begin{equation}\label{e:Garding0}\forall u\in L^2(\IR^2\times\IS^1),\quad\langle \Op_h^w(a)u,u\rangle\geq -C_a h\|u\|_{L^2}^2.
 \end{equation}

\bibliographystyle{alpha}
\bibliography{allbiblio}

\end{document}